\documentclass[a4paper, 11pt,reqno]{amsart}

\usepackage[T1]{fontenc}      
\usepackage[foot]{amsaddr}    

\usepackage[british]{babel}   
\usepackage{csquotes}         
\usepackage[babel]{microtype} 
\usepackage{enumerate}

\usepackage{mathtools}        
\usepackage{amssymb}          
\usepackage{amsthm}           
\usepackage{amstext}          
\usepackage{mleftright}       
\usepackage[makeroom]{cancel} 
\usepackage{mathdots}         
\usepackage{mathrsfs}         

\usepackage{stickstootext}
\usepackage[stickstoo,varbb]{newtxmath} 
\makeatletter
\DeclareFontFamily{U}{ntxmia}{\skewchar \font =127}
 \DeclareFontShape{U}{ntxmia}{m}{it}{
                        <-> \ntxmath@scaled ntxmia
                      }{}    
                      \DeclareFontShape{U}{ntxmia}{b}{it}{
                        <-> \ntxmath@scaled ntxbmia
                      }{}
\makeatother

\usepackage{graphicx}         
\usepackage{psfrag}           
\usepackage{caption}          
\usepackage{tikz}             
\usetikzlibrary{backgrounds}

\usepackage{booktabs}         

\usepackage{enumitem}         

\usepackage[margin=1in]{geometry} 

\usepackage[textsize=footnotesize]{todonotes}

\usepackage[numbers,sort]{natbib}

\makeatletter
\def\NAT@spacechar{~}
\makeatother

\usepackage{hyperref}              
\usepackage[capitalise]{cleveref}  
\hypersetup{colorlinks=true,
   citecolor=blue,
   filecolor=blue,
   linkcolor=blue,
   urlcolor=blue
}
\usepackage{url}

\makeatletter
\if@cref@capitalise
\crefname{figure}{figure}{figures}
\crefname{claim}{Claim}{Claims}
\else
\crefname{figure}{Figure}{Figures}
\crefname{claim}{claim}{claims}
\fi
\makeatother
\Crefname{figure}{Figure}{Figures}
\Crefname{claim}{Claim}{Claims}

\usepackage{algorithm}
\usepackage{algpseudocode}

\allowdisplaybreaks

\newtheoremstyle{primetheorem}
{\topsep}                
{}                
{\slshape}        
{}                
{\bfseries}       
{'.}               
{5pt plus 1pt minus 1pt}               
{}                

\theoremstyle{definition}
\newtheorem{definition}{Definition}
\newtheorem{remark}[definition]{Remark}

\theoremstyle{plain}
\newtheorem{claim}{Claim}

\newtheorem{theorem}[definition]{Theorem}

\newtheorem{lemma}[definition]{Lemma}

\newtheorem{problem}[definition]{Problem}

\theoremstyle{primetheorem}
\newtheorem{primetheorem}{Theorem}

\newenvironment{claimproof}{%
\let\origqed=\qedsymbol%
\renewcommand{\qedsymbol}{$\blacktriangleleft$}%
\begin{proof}}{\end{proof}\let\qedsymbol=\origqed}

\numberwithin{equation}{section}

\renewcommand{\binom}[2]{\ensuremath{\mleft(\kern-.1em\genfrac{}{}{0pt}{}{#1}{#2}\kern-.1em\mright)}}    
\newcommand{\inbinom}[2]{\ensuremath{\bigl(\kern-.1em\genfrac{}{}{0pt}{}{#1}{#2}\kern-.1em\bigr)}} 

\makeatletter
\def\moverlay{\mathpalette\mov@rlay}
\def\mov@rlay#1#2{\leavevmode\vtop{%
  \baselineskip\z@skip \lineskiplimit-\maxdimen
  \ialign{\hfil$\m@th#1##$\hfil\cr#2\crcr}}}
\newcommand{\charfusion}[3][\mathord]{
    #1{\ifx#1\mathop\vphantom{#2}\fi
        \mathpalette\mov@rlay{#2\cr#3}
      }
    \ifx#1\mathop\expandafter\displaylimits\fi}
\makeatother

\newcommand{\COMMENT}[1]{}
\newcommand{\COMNEW}[1]{}
\renewcommand{\COMNEW}[1]{\footnote{\textcolor{red!70!black}{#1}}} 

\title{Long running times for hypergraph bootstrap percolation}

\author[A.~Espuny D\'iaz]{Alberto Espuny D\'iaz}
\email{alberto.espuny-diaz@tu-ilmenau.de}
\address[Espuny D\'iaz]{Institut f\"ur Mathematik, Technische Universit\"at Ilmenau, 98684 Ilmenau, Germany.}

\author[B.~Janzer]{Barnab\'as Janzer}
\email{bkj21@cam.ac.uk}
\address[Janzer]{Department of Pure Mathematics and Mathematical Statistics, University of Cambridge, Wilberforce Road, Cambridge CB3 0WB, United Kingdom.}

\author[G.~Kronenberg]{Gal Kronenberg}
\email{kronenberg@maths.ox.ac.uk}
\address[Kronenberg]{Mathematical Institute, University of Oxford, Andrew Wiles Building, Radcliffe Observatory Quarter, Woodstock Road, Oxford, United Kingdom.}

\author[J.~Lada]{Joanna Lada}
\email{j.m.lada@lse.ac.uk}
\address[Lada]{London School of Economics and Political Science,
Houghton Street, London, WC2A 2AE, United Kingdom}

\thanks{The research leading to these results started during the \emph{Young Researchers in Combinatorics} workshop hosted by ICMS.
We are indebted to the ICMS and the organisers of the event.\\
A.~Espuny Díaz was partially funded by the Carl-Zeiss-Foundation and by DFG grant PE 2299/3-1.\\
B.~Janzer was funded by an EPSRC DTP Scholarship.\\
G.~Kronenberg was funded by the European Union’s Horizon
2020 research and innovation programme under the Marie Sk\l odowska Curie grant  No. 101030925.\\
J.~Lada was funded by an LSE PhD Studentship.}

\date{\today}

\begin{document}

\begin{abstract}
Consider the hypergraph bootstrap percolation process in which, given a fixed $r$-uniform hypergraph $H$ and starting with a given hypergraph $G_0$, at each step we add to $G_0$ all edges that create a new copy of $H$.
We are interested in maximising the number of steps that this process takes before it stabilises.
For the case where $H=K_{r+1}^{(r)}$ with $r\geq3$, we provide a new construction for $G_0$ that shows that the number of steps of this process can be of order $\Theta(n^r)$.
This answers a recent question of Noel and Ranganathan. 
To demonstrate that different running times can occur, we also prove that, if $H$ is $K_4^{(3)}$ minus an edge, then the maximum possible running time is $2n-\lfloor \log_2(n-2)\rfloor-6$.
However, if $H$ is $K_5^{(3)}$ minus an edge, then the process can run for $\Theta(n^3)$ steps.

\end{abstract}

\maketitle

\section{Introduction}

The \emph{hypergraph bootstrap percolation process} is an infection process on hypergraphs which was introduced by Bollob\'as in 1968 under the name of \emph{weak saturation}~\cite{Bollobas68}.
For an integer $r\geq2$ and a set $S$, denote by $\inbinom{S}{r}$ the set of all subsets of $S$ of size $r$.
Given an $r$-uniform hypergraph $H$ and a positive integer $n$, the \emph{$H$-bootstrap percolation process} is a deterministic process defined as follows.
We start with a given $r$-uniform hypergraph $G_0$ on vertex set $[n]\coloneqq\{1,\ldots,n\}$.
For each time step $t\geq1$, we define the hypergraph $G_t$ on the same vertex set $[n]$ by letting
\[E(G_t)\coloneqq E(G_{t-1})\cup \left\{e\in \binom{[n]}{r} : \exists \text{ an $H$-copy $H'$ s.t. $E(H')\nsubseteq E(G_{t-1})$ and $E(H')\subseteq E(G_{t-1})\cup\{e\}$}  \right\},\]
that is, $G_t$ is an $r$-uniform hypergraph on $[n]$ defined by including all edges of $G_{t-1}$ together with all edges $e\in\inbinom{[n]}{r}$ which create a new copy of $H$ with the edges of $G_{t-1}$.
The hypergraph $G_0$ is called the \emph{initial infection}, and the edges $E(G_t)\setminus E(G_{t-1})$ are said to be \emph{infected} at time $t$.
If there exists some $T\geq 0$ such that $G_T=K_n^{(r)}$, we say that $G_0$ \emph{percolates} under this process.
In the weak saturation interpretation, we say that the hypergraph $G_0$ is weakly $H$-saturated if $G_0$ is $H$-free and percolates under $H$-bootstrap percolation, that is, if there exists an ordering of $E(K_n^{(r)})\setminus E(G_0) = \{e_1,\dots,e_t\}$ such that the addition of $e_i$ to $G_0\cup \{e_1,\dots ,e_{i-1}\}$ will create a new copy of $H$, for every $i \in [t]$.

Given a fixed hypergraph $H$, one of the most studied extremal problems in this setting is establishing  the minimum size of an $n$-vertex hypergraph which is weakly $H$-saturated.
For the most basic case, where $r=2$ and $H=K_k$, it was conjectured by Bollob\'as~\cite{Bollobas68} that the minimum size of a weakly $K_k$-saturated $n$-vertex graph is $(k-2)n - \inbinom{k-1}{2}$.
This was independently confirmed by \citet{Alon85}, \citet{Frankl82} and \citet{kalai1985hyperconnectivity,Kalai84} using methods from linear algebra.
For the hypergraph case, the work of \citet{Frankl82} also settles this problem for $K_k^{(r)}$ with $r\geq3$.
This problem has also been studied for other graphs $H$, and for host graphs other than the complete graph, and other related settings; see, e.g.,~\cite{bulavka2021weak,moshkovitz2015exact,shapira2021weakly,Alon85,Pik01a,Pik01b,kronenberg2021weak}.

Even though the initial infection graphs which are solutions to the weak saturation problem have the minimum possible number of edges, it is interesting to note that, in all the known examples, they require only very few steps until the infection process stabilises.
For example, the weakly $K_k$-saturated graph achieving the minimum size is given by removing the edges of a clique of size $n-k+2$ from $K_n$, which means that only one step is needed in order to complete the infection process.
In this direction, Bollob\'as raised the problem of finding the initial infection for which the running time of the $H$-bootstrap percolation process is maximised.
This was previously studied in the related setting of neighbourhood percolation by Benevides and Przykucki~\cite{benevides2013slowly,benevides2015maximum,przykucki2012maximal}, and in the random setting by Gunderson, Koch and Przykucki~\cite{gunderson2017time}.

Here we consider this problem in the hypergraph bootstrap percolation setting.
Given a fixed $r$-uniform hypergraph $H$ and an $r$-uniform initial infection $G_0$, we define the \emph{running time} of the $H$-bootstrap percolation process on $G_0$ to be
\[M_H(G_0) \coloneqq \min\{t\geq0:G_t=G_{t+1}\}.\]
We denote the \emph{maximum running time} over all $r$-uniform hypergraphs $G_0$ on $n$ vertices as $M_H(n)$.
We shall simplify these notations to $M^r_k(G_0)$ and $M^r_k(n)$ when $H=K^{(r)}_k$ is the complete $r$-uniform hypergraph on $k$ vertices, and drop the superscript to $M_k(n)$ in the graph setting ($r=2$).
Note that a trivial upper bound for $M_H(n)$ is given by $\binom{n}{r}$, the total number of edges of $K_n^{(r)}$.

The simplest setting to consider is for graph bootstrap percolation and $H=K_k$.
For $k=3$, it is not hard to see that $M_3(n) = \lceil\log_2(n-1)\rceil$, where an extremal example is given by an $n$-path (see, e.g., \cite{BPRS17} for the details).
\citet{BPRS17} and independently \citet{Mat15} considered this problem for higher values of $k$.
By carefully analysing the growth of cliques during the percolation process, both groups of authors showed that $M_4(n)=n-3$. 
Moreover, for $k\geq 5$, \citet{BPRS17} obtained the lower bound $M_k(n)\geq n^{2-\alpha_k-o(1)}$, where $\alpha_k = ({k-2})/({\inbinom{k}{2}-2})$, using a probabilistic argument.
The authors of \cite{BPRS17} conjectured that $M_k(n)=o(n^2)$ for all $k\geq 5$.
However, in a subsequent paper, \citet{BKPS19} disproved this conjecture for $k\geq 6$, showing that the natural upper bound is tight up to a constant factor.
The authors of \cite{BKPS19} also improved the lower bound for $k=5$ to $M_5(n) \geq n^{2-O(1/\sqrt{\log n})}$, using Behrend's construction of `dense' 3-AP-free sets, and conjectured that $M_5(n)=o(n^2)$.
It remains an open problem to determine whether this is the case. 

In this paper we consider the question of the maximum running time when $H$ is an $r$-uniform hypergraph with $r\geq 3$.
This was recently investigated by \citet{NR22}.
By providing an explicit construction to establish the lower bound (noting the trivial upper bound of $\binom{n}{r}$), they proved the following theorem for the case $k\geq r+2$.

\begin{theorem}[Noel and Ranganathan \cite{NR22}]
Let\/ $r\geq 3$.
If\/ $k\geq r+2$, then\/ $M^r_k(n) = \Theta(n^r)$.
\end{theorem}

For the case $k=r+1$, they established the following lower bound.

\begin{theorem}[Noel and Ranganathan \cite{NR22}]
Let\/ $r\geq 3$.
If\/ $k=r+1$, then\/ $M^r_k(n) = \Omega(n^{r-1})$.
\end{theorem}

This theorem leaves a gap between the lower bound and the trivial upper bound $M^r_{r+1}(n) = O(n^r)$.
Noel and Ranganathan conjectured that $M^3_{4}(n) = O(n^2)$~\cite[Conjecture~5.1]{NR22}, but suggested that, for sufficiently large $r$, it is indeed true that the maximum running time achieves $M^r_{r+1}(n) = \Theta(n^r)$~\cite[Question~5.2]{NR22}. 

In this paper, we show the conjecture to be false and prove that the trivial upper bound is in fact tight, up to a constant factor, for all $r\geq 3$.
This also gives a positive answer to their question, in a strong sense.

\begin{theorem}\label{theorem_bootstrapgeneral}
	For any fixed integer\/ $r\geq 3$, we have\/ $M_{r+1}^r(n)=\Theta(n^r)$.
\end{theorem}

Another proof for \cref{theorem_bootstrapgeneral} was independently announced by \citet{HL22}.

We note that \cref{theorem_bootstrapgeneral} establishes a clear difference with respect to the graph case $r=2$, where $M_k(n)=o(n^r)$ for $k\in\{r+1,r+2\}$ (and possibly also $r+3$).
It may therefore seem that the behaviour of hypergraph bootstrap percolation is less rich than its graph counterpart. We propose a modification of the problem above that shows this is not the case, and that different (and very interesting) asymptotic running times may still occur in the hypergraph setting.

Indeed, recall that we may think of $H$-bootstrap percolation as an infection process where the infection spreads to a new copy of $H$ if only one edge of said copy was not infected in the previous step.
It is reasonable then to consider models where the infection is more powerful, in the sense that it will extend to copies of $H$ which are missing at most $m$ edges, for some fixed integer $m$.
We consider here in particular the case $m=2$. 
Note that if $m=2$ and $H$ is a complete hypergraph (which is the case we will focus on), then this modified model is equivalent to the original hypergraph percolation process for the hypergraph $H'$ obtained by deleting an arbitrary edge from $H$.

Formally, let $H$ be a given $r$-uniform hypergraph, and let $G$ be an $r$-uniform hypergraph on $[n]$.
For each copy $H'$ of $H$ on $[n]$, if $|E(H')\setminus E(G)|\leq m$, we say that $H'$ is \emph{$m$-completable} in $G$.
We define the $(H,m)$-bootstrap percolation process on an initial infection $G_0$ on $[n]$ to be the sequence of hypergraphs $G_0,G_1,\ldots$ on $[n]$ given by setting, for each $t\geq1$,
\[E(G_t)\coloneqq\bigcup_{\substack{H'\text{ copy of }H\text{ on }[n]\\H'\ m\text{-completable in }G_{t-1}}}E(H').\]
Note that the $(H,1)$-bootstrap percolation process simply corresponds to the usual $H$-bootstrap percolation process.
Let us denote the running time of this hypergraph percolation process as $M_{(H,m)}(G_0)\coloneqq\min\{t\geq0:G_t=G_{t+1}\}$, and the maximum running time over all $r$-uniform $n$-vertex hypergraphs $G_0$ as $M_{(H,m)}(n)$.
The next result shows that we get interesting new behaviour when $m=2$ and $H=K_4^{(3)}$ (which is probably the most natural first case to consider).

\begin{theorem}\label{prop:linear}
For all\/ $n\geq 4$, we have\/ $M_{(K_{4}^{(3)},2)}(n)=2n-\lfloor\log_2(n-2)\rfloor-6$.
\end{theorem}

It is worth remarking here that this is the first nontrivial exact result about running times of hypergraph bootstrap percolation.
The only nontrivial exact results in graph bootstrap percolation are those for $K_3$- and $K_4$-bootstrap percolation~\cite{BPRS17}.

We also prove that in the next case, $H=K_5^{(3)}$, the running time can once again be cubic (i.e., as large as possible).

\begin{theorem}\label{thm:cubic}
We have\/ $M_{(K_{5}^{(3)},2)}(n)=\Theta(n^3)$.
\end{theorem}

Let ${K}_s^{(r)}-e$ denote the hypergraph obtained by deleting an edge from ${K}_s^{(r)}$.
As mentioned above, the $({K}_s^{(r)},2)$-process is the same as the usual bootstrap percolation process for ${K}_s^{(r)}-e$, so the results above can be reformulated as follows.\medskip

\setcounter{primetheorem}{3}
\begin{primetheorem}\label{thmp:1}
For all\/ $n\geq 4$, we have\/ $M_{{K}_{4}^{(3)}-e}(n)=2n-\lfloor\log_2(n-2)\rfloor-6$.
\end{primetheorem}

\begin{primetheorem}\label{thmp:2}
We have\/ $M_{{K}_{5}^{(3)}-e}(n)=\Theta(n^3)$.
\end{primetheorem}

We present our proof of \cref{theorem_bootstrapgeneral} in \cref{sect2}.
We defer the proofs of \cref{prop:linear,thm:cubic} to \cref{sect3}.
We also propose some open problems in our concluding remarks.

\section{Long running times for simple infections}\label{sect2}

In order to prove \cref{theorem_bootstrapgeneral}, we will use a result of Noel and Ranganathan~\cite{NR22} that allows us to focus on the case $r=3$.
To state their result, we need to recall some definitions from~\cite{NR22}.
Let $G_0$ be an $r$-uniform hypergraph, let $G_t$ be the hypergraph at time $t$ for the $K_{r+1}^{(r)}$-bootstrap process starting with $G_0$ as initial infection, and let $T=M_{r+1}^r(G_0)$ be the time the process stabilises.
We say that $G_0$ is \emph{$K_{r+1}^{(r)}$-civilised} if the following conditions are satisfied for some edge $e_0$ of $G_0$.
\begin{enumerate}
	\item For each $t\in[T]$, $G_t$ contains only one more edge $e_t$ than $G_{t-1}$, and one more copy $H_t$ of $K_{r+1}^{(r)}$.
	\item For all $t\in[T]$ we have $E(H_t)\cap\{e_0,e_1,\dots,e_T\}=\{e_{t-1},e_t\}$.
	\item The $K_{r+1}^{(r)}$-bootstrap percolation process starting with $G_0-e_0$ infects no edge.
\end{enumerate}

\begin{lemma}[Noel and Ranganathan~{\cite[Lemma~2.11]{NR22}}] \label{lemma_r=3}
If for all\/ $n$ there exists a\/ $K_4^{(3)}$-civilised hypergraph\/ $G_0$ on\/ $\Theta(n)$ vertices such that\/ $M_4^3(G_0)=\Theta(n^3)$, then for all\/ $r\geq 3$ we have\/ $M_{r+1}^r(n)=\Theta(n^r)$.
\end{lemma}

Before we give the formal proof of Theorem~\ref{theorem_bootstrapgeneral}, let us give an informal description of the construction that gives a lower bound for the number of steps of the percolation process.
As noted above, by \cref{lemma_r=3} it is enough to consider the case $r=3$.
The main part of the construction consists of three layers of vertices: `top' vertices labelled $t_i$, `bottom' vertices labelled $b_j$, and `middle' vertices labelled $m_\ell$.
In each time step, just one new edge will become infected.
That infection will happen because one copy of $K_4^{(3)}$, which had only two edges present in the initial infection, has a third edge infected in the previous step of the process.

The process will  consist mainly of chains of infections, where we move from one chain to another by using special gadgets.
The chains will have the format of the so-called `beachball hypergraph'.
The vertex set of this hypergraph consists of one top and one bottom vertex, and some ordered vertices in the middle; the edges are the triples consisting of two consecutive middle vertices, and either the top or the bottom vertex. 
See \Cref{fig:M1} for an illustration.

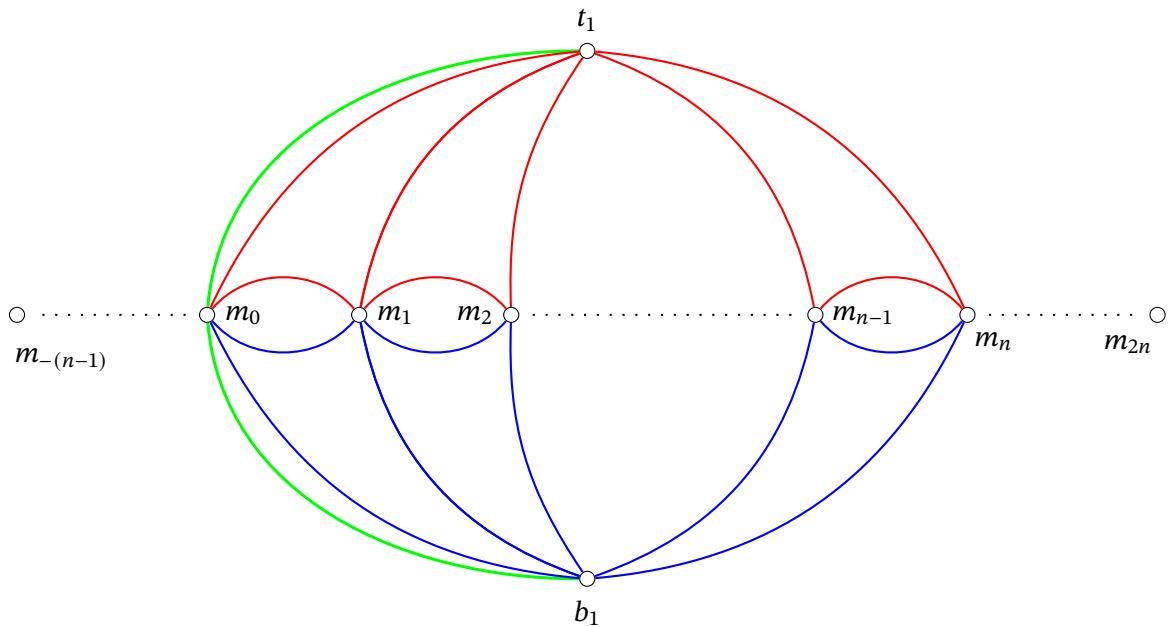
\begin{figure}[b]
\centering    
\begin{tikzpicture}

    
    \begin{scope}[every node/.style={circle, draw, inner sep=2pt}]
        \node[label={0:$m_0$}] (v1) at (0,0) {};
        \node[label={0:$m_1$}] (v2) at (2,0) {};
        \node[label={180:$m_2$}] (v3) at (4,0) {};
        \node[label={0:$m_{n-1}$}] (v8) at (8,0) {};
        \node[label={-45:$m_{n}$}] (vn) at (10,0) {};
        \node[label={90:$t_{1}$}] (t1) at (5,3.5) {};
        \node[label={270:$b_{1}$}] (b1) at (5,-3.5) {};
    \end{scope}
    
    
    \begin{scope}[every node/.style={circle, draw, inner sep=2pt}]
        \node[label={225:$m_{2n}$}] (v2n) at (12.5,0) {};
        \node[label={-45:$m_{-(n-1)}$}] (v-n+1) at (-2.5,0) {};
    \end{scope}
    
    
    \begin{scope}[thick]
        \draw[loosely dotted, shorten >=5pt, shorten <=5pt] (v3) -- (v8);
        
        \draw[loosely dotted, shorten >=5pt, shorten <=5pt] (vn) -- (v2n);
        \draw[loosely dotted, shorten >=5pt, shorten <=5pt] (v1) -- (v-n+1);
        
        \draw[line width=0.4mm, green] (t1) to[out=180, in=85] (v1);
        \draw[line width=0.4mm, green] (b1) to[out=180, in=275] (v1);

        \draw[red] (t1) to[bend left] (vn);

        \draw[red] (t1) to[bend right] (v1);
        \draw[red] (t1) to[bend right] (v2);
        \draw[red] (t1) to[bend right] (v2);
        \draw[red] (t1) to[bend right=18] (v3);
        \draw[red] (t1) to[bend left] (v8);

        \draw[red] (v1) to[bend left=50] (v2);
        \draw[red] (v2) to[bend left=50] (v3);
        \draw[red] (v8) to[bend left=50] (vn);

        \draw[blue] (b1) to[bend right] (vn);

        \draw[blue] (b1) to[bend left] (v1);
        \draw[blue] (b1) to[bend left] (v2);
        \draw[blue] (b1) to[bend left] (v2);
        \draw[blue] (b1) to[bend left=18] (v3);
        \draw[blue] (b1) to[bend right] (v8);

        \draw[blue] (v1) to[bend right=50] (v2);
        \draw[blue] (v2) to[bend right=50] (v3);
        \draw[blue] (v8) to[bend right=50] (vn);
    \end{scope}
\end{tikzpicture}
\caption{Initial infection $G_0$ showing only the first top and bottom vertices, $t_1$ and $b_1$.
Each red or blue triangle represents an edge of $G_0$, and together they form the first beachball hypergraph in our process.
The green arc represents an edge containing the vertices it passes through.
To form $G_1$, the edge $t_1m_1b_1$ is added, as this completes a copy of $K^{(3)}_4$ on $\{t_1,m_0,m_1,b_1\}$.
It is clear to see that subsequently all edges of the form $t_1m_ib_1$ for $i$ increasing from $2$ to $n$ are added in turn.}
\label{fig:M1}
\end{figure}

It will be convenient to think of the process as having $n$ phases, each phase having $\Theta(n)$ stages, and each stage having $\Theta(n)$ infection steps.
A phase will represent the infection process that occurs when we fix a top vertex $t_i$.
In each phase, we have $\Theta(n)$ stages, where each stage is the process that occurs when we fix $b_j$ (for the fixed $t_i$ of this phase).
At a specified phase and stage, the initial infected set will be the above mentioned beachball hypergraph, and the infection will spread through the middle vertices.
This gives $\Theta(n)$ infection steps for each stage.

The challenge will then be to move to another top or bottom vertex without infecting more than one edge in each step of the process.
For this purpose, we will introduce, at the end of each stage, a new middle vertex and a special `switching' gadget.
Each stage of the process will be represented by a tuple  of a top vertex $t_i$, a bottom vertex $b_j$, and consecutive middle vertices starting from $m_s$ and ending in $m_\ell$, where $-(n-1)\leq s\leq 0$ and $n\leq \ell\leq 2n$.
For moving between phases, we will introduce a different type of gadget.

Let us first describe the first few stages of the process to give a better intuition.
The first $n$ infection steps will come from a `path' on the middle layer: the edges $t_1m_\ell m_{\ell +1}$ and $b_1m_\ell m_{\ell +1}$ will be present at time zero for all $0\leq \ell \leq n-1$, as well as the edge $t_1b_1m_0$. So once the edge $t_1b_1m_\ell$ becomes infected, it propagates in the next step to $t_1b_1m_{\ell +1}$.
See Figure~\ref{fig:M1} for an illustration.

After $\Theta(n)$ such infections, we want to swap out $b_1$ to another bottom vertex (labelled $b_{-1}$).
We do this by making sure that the last infected edge using $b_1$ (namely, the edge $t_1b_1m_n$) makes the middle path longer, that is, it makes the edge $t_1m_nm_{n+1}$ infected in the next step.
To achieve this, we will have $b_1m_nm_{n+1}$ and $t_1b_1m_{n+1}$ present in the original hypergraph $G_0$; see Figure~\ref{fig:M2}.

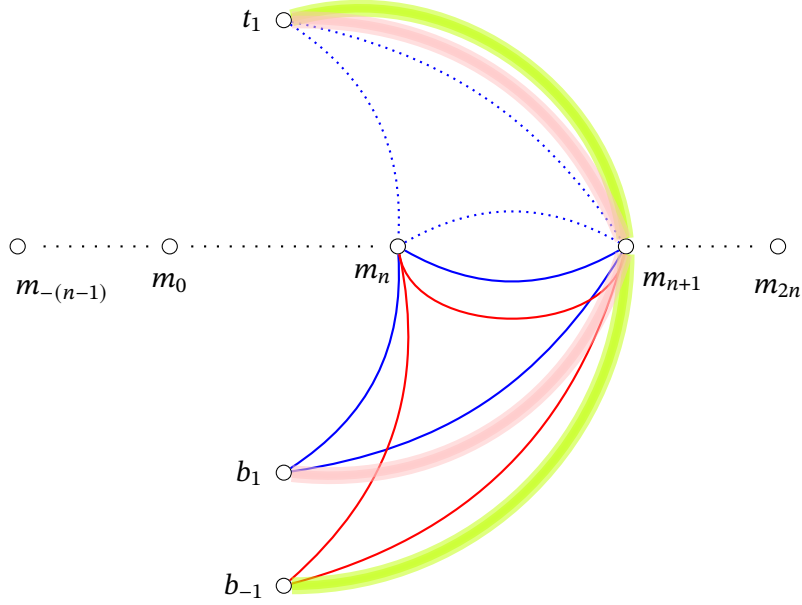
\begin{figure}[t]
\centering    

\begin{tikzpicture}
    
    \begin{scope}[every node/.style={circle, draw, inner sep=2pt}]
        \node[label={270:$m_0$}] (m0) at (4,0) {};
        
        \node[label={235:$m_{n}$}] (mn) at (7,0) {};

        \node[label={180:$t_{1}$}] (t1) at (5.5,3) {}; 
        \node[label={180:$b_{1}$}] (b1) at (5.5,-3) {};

        \node[label={180:$b_{-1}$}] (b-1) at (5.5,-4.5) {};
    \end{scope}
    
    
    \begin{scope}[every node/.style={circle, draw, inner sep=2pt}]
        \node[label={340:\ \ $m_{n+1}$}] (mn+1) at (10,0) {};
        \node[label={270:$m_{2n}$}] (m2n) at (12,0) {};
        \node[label={315:$m_{-(n-1)}$}] (m-n+1) at (2,0) {};
    \end{scope}
    
    
    \begin{scope}[thick]
        \draw[loosely dotted, shorten >=5pt, shorten <=5pt] (m0) -- (mn);
        
        \draw[loosely dotted, shorten >=5pt, shorten <=5pt] (mn+1) -- (m2n);
        \draw[loosely dotted, shorten >=5pt, shorten <=5pt] (m0) -- (m-n+1);
        
        \draw[blue] (b1) to[bend right=25] (mn+1);
        \draw[blue] (mn) to[bend right] (mn+1);
        \draw[blue] (b1) to[bend right] (mn);

        \draw[red] (b-1) to[bend right] (mn+1);
        \draw[red] (b-1) to[bend right] (mn);
        \draw[red] (mn) to[bend right=80] (mn+1);
        
        \draw[dotted, blue] (t1) to[bend left] (mn);
        \draw[dotted, blue] (t1) to[bend left=22] (mn+1);
        \draw[dotted, blue] (mn) to[bend left] (mn+1);
        
       \draw[lime, ultra thick, preaction= {draw, lime, -, double=lime, double distance=2\pgflinewidth,opacity=0.5},opacity=0.1] (t1) to[bend left=50] (mn+1);
        \draw[pink, ultra thick, preaction= {draw, pink, -, double=pink, double distance=2\pgflinewidth,opacity=0.5},opacity=0.1](t1) to[bend left=35] (mn+1);
       \draw[lime, ultra thick, preaction= {draw, lime, -, double=lime, double distance=2\pgflinewidth,opacity=0.5},opacity=0.1] (mn+1) to[bend left=45] (b-1);
        \draw[pink, ultra thick, preaction= {draw, pink, -, double=pink, double distance=2\pgflinewidth,opacity=0.5},opacity=0.1] (mn+1) to[bend left=40] (b1);
    \end{scope}
\end{tikzpicture}
\caption{Switching gadget to change from $K^{(3)}_4$ copies containing $b_1$ to those containing $b_{-1}$. The edges $b_1m_nm_{n+1}$ and $b_{-1}m_nm_{n+1}$ are present in the initial infection $G_0$.
After the edge $t_1m_nb_1$ is created by the percolation process, the copy of $K^{(3)}_4$ induced by the vertices $\{t_1, m_n,m_{n+1},b_1\}$ is present, except for the missing edge $t_1m_nm_{n+1}$ shown in the dotted blue line.
Thus this edge is added, followed by $t_1m_nb_{-1}$.
This triggers the process to run backwards and create all edges of form $t_1m_ib_{-1}$, for $i$ decreasing from $i=n-1$ to $i=0$, in turn.}
\label{fig:M2}
\end{figure}

Once $t_1m_nm_{n+1}$ is infected, it can start a chain of infections using the new bottom vertex $b_{-1}$.
However, this time the chain of infections will go in the opposite direction on the middle path: we will first infect $t_1b_{-1}m_n$ (for this we will need the edges $t_1b_{-1}m_{n+1}$ and $b_{-1}m_nm_{n+1}$ to be present initially, as in \Cref{fig:M2}), then we infect $t_1b_{-1}m_{n-1}$, and so on, until $t_1b_{-1}m_0$.

At this point we again swap out the bottom vertex to a different one (labelled $b_2$)--- we do this using the same trick as above, i.e., making the middle path one longer, and then changing the direction we traverse the path.
We keep repeating the steps above for $\Theta(n)$ bottom vertices to get $\Theta(n^2)$ infections which all use the same top vertex $t_1$.

Once we have the $\Theta(n^2)$ infections using $t_1$, we wish to swap out the top vertex $t_1$ to a different one (labelled $t_2$).
We could do this similarly to how we swapped the $b_j$'s, but it is more convenient to simply introduce a gadget using three `dummy' vertices $d_1, d_2, d_3$ to do this swap.
The last infection using $t_1$ (namely, $t_1m_{2n-1}m_{2n}$) will start a short chain of infections using the $K_4^{(3)}$'s given by $t_1m_{2n-1}m_{2n}d_1$, $m_{2n-1}m_{2n}d_1d_2$, $m_{2n}d_1d_2d_3$, $d_1d_2d_3t_2$, $d_2d_3t_2m_0$, and $d_3t_2m_0m_1$.
The last one of these allows us to start a repeat of the previous infection process, using $t_2$ instead of $t_1$.
We will use three such dummy vertices $d_i$ for each of the $n-1$ swaps at the top--- so only $3n-3=\Theta(n)$ dummy vertices in total.
See \Cref{fig:M4} for an illustration. \\

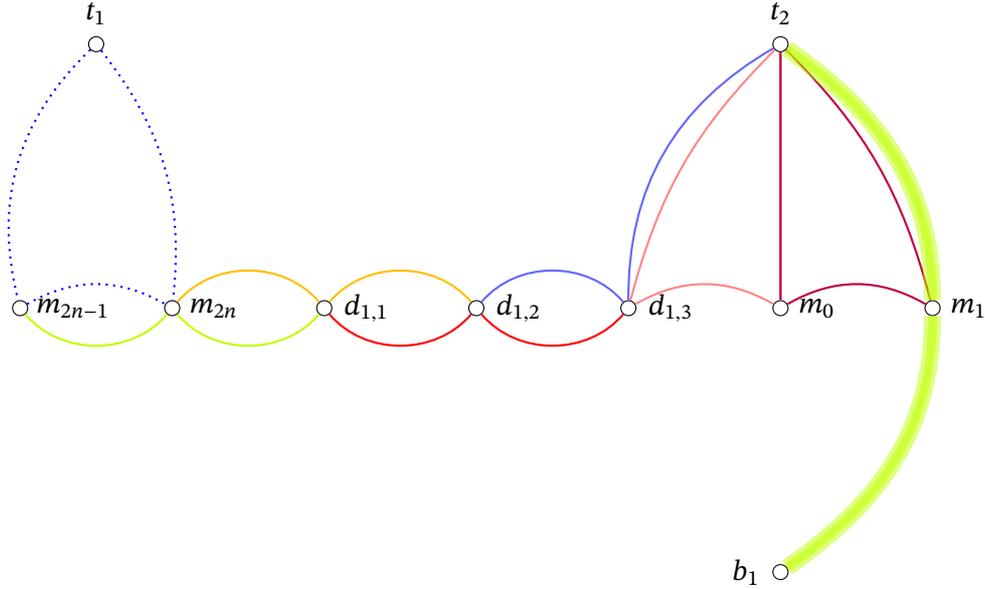
\begin{figure}
\centering    

\begin{tikzpicture}

    
    \begin{scope}[every node/.style={circle, draw, inner sep=2pt}]
        
        \node[label={0:$m_{2n-1}$}] (mn) at (0,0) {};
        \node[label={0:$m_{2n}$}] (mn+1) at (2,0) {};
        \node[label={0:$d_{1,1}$}] (d1) at (4,0) {};
        \node[label={0:$d_{1,2}$}] (d2) at (6,0) {};
        \node[label={0:$d_{1,3}$}] (d3) at (8,0) {};
        \node[label={0:$m_{0}$}] (m0) at (10,0) {};
        \node[label={0:$m_{1}$}] (m1) at (12,0) {};
        \node[label={90:$t_{1}$}] (t1) at (1,3.5) {}; 
        \node[label={90:$t_{2}$}] (t2) at (10,3.5) {};
        
        \node[label={180:$b_{1}$}] (b1) at (10,-3.5) {};
    \end{scope}
    
    
    \begin{scope}[every node/.style={circle, draw, inner sep=2pt}]
    \end{scope}
    
    
    \begin{scope}[thick]
        \draw[dotted, blue] (t1) to[bend right] (mn);
        \draw[dotted, blue] (t1) to[bend left=22] (mn+1);
        \draw[dotted, blue] (mn) to[bend left] (mn+1);
        
       \draw[lime] (mn) to[bend right=50] (mn+1);
       \draw[lime] (mn+1) to[bend right=50] (d1);
       \draw[orange!50!yellow] (mn+1) to[bend left=50] (d1);
       \draw[orange!50!yellow] (d1) to[bend left=50] (d2);
       \draw[red] (d1) to[bend right=50] (d2);
       \draw[red] (d2) to[bend right=50] (d3);
       \draw[blue!60!white] (d2) to[bend left=50] (d3);
       \draw[blue!60!white] (d3) to[bend left] (t2);
       \draw[pink!70!red] (d3) to[bend left=15] (t2);
       \draw[pink!70!red] (d3) to[bend left] (m0);
       \draw[pink!70!red] (t2) to[in =90, out=-90] (m0);
       \draw[purple] (m0) to[in =-90, out=90] (t2);
       \draw[purple] (m0) to[bend left] (m1);
       \draw[purple] (t2) to[bend left=15] (m1);
       
       \draw[lime, ultra thick, preaction= {draw, lime, -, double=lime, double distance=2\pgflinewidth,opacity=0.5},opacity=0.1] (t2) to[in=90, out= -35] (m1);
       \draw[lime, ultra thick, preaction= {draw, lime, -, double=lime, double distance=2\pgflinewidth,opacity=0.5},opacity=0.1] (m1) to[in= 35, out= -90] (b1);
    \end{scope}
\end{tikzpicture}
\caption{Switching gadget for changing the top vertex $t_1$ to $t_2$.
Edges present in the initial infection $G_0$ are omitted for clarity.
When the dotted blue edge $t_1m_{2n-1}m_{2n}$ is infected, this causes the edges along the chain to become infected, ending in $t_2m_0m_1$.
This triggers the infection of $t_2m_1b_1$, and in turn the process from the stage as shown in Figure~\ref{fig:M1}, with $t_1$ replaced with $t_2$.}
\label{fig:M4}
\end{figure}

Let us now turn to the formal proof of our theorem.

\begin{proof}[Proof of Theorem~\ref{theorem_bootstrapgeneral}]
By Lemma~\ref{lemma_r=3}, it suffices to consider the case $r=3$ and show that there are $K_4^{(3)}$-civilised $3$-uniform hypergraphs on $\Theta(n)$ vertices such that the $K_4^{(3)}$-bootstrap process takes $\Theta(n^3)$ steps to stabilise.
We now describe a construction achieving this.

The initial infection hypergraph $G_0$ has $9n-4=\Theta(n)$ vertices, which are labelled as follows: $t_1,\dots,t_n$, $b_1,\dots,b_n,b_{-1},\dots,b_{-(n-1)}$, $m_{-(n-1)}, m_{-(n-2)}\dots,m_{2n}$, and $d_{i,1},d_{i,2},d_{i,3}$ for $i\in[n-1]$.
The edges of $G_0$ are given below:
\begin{enumerate}[label=(\alph*)]
	\item\label{edges1} $t_1m_0m_1$;
	\item\label{edges2} $t_im_\ell m_{\ell +1}$ for all $i\in[n]$ and $\ell\in[n-1]$;
	\item\label{edges3} $b_jm_\ell m_{\ell +1}$ for all $j\in[n]$ and $\ell\in[-(j-1),n+j-1]$;
	\item\label{edges4} $b_{-j}m_\ell m_{\ell +1}$ for all $j\in[n-1]$ and $\ell\in[-j,n+j-1]$;
	\item\label{edges5} $t_ib_{j}m_{-(j-1)}$ and $t_ib_jm_{n+j}$ for all $i,j\in[n]$;
	\item\label{edges6} $t_ib_{-j}m_{n+j}$ and $t_ib_{-j}m_{-j}$ for all $i\in[n]$ and $j\in[n-1]$;
	\item\label{edges7} $t_im_{2n-1}d_{i,1}$, $t_im_{2n}d_{i,1}$, $m_{2n-1}m_{2n}d_{i,2}$, $m_{2n-1}d_{i,1}d_{i,2}$, $m_{2n}d_{i,1}d_{i,3}$, $m_{2n}d_{i,2}d_{i,3}$, $d_{i,1}d_{i,2}t_{i+1}$,\linebreak{} $d_{i,1}d_{i,3}t_{i+1}$, $d_{i,2}d_{i,3}m_0$, $d_{i,2}t_{i+1}m_0$, $d_{i,3}t_{i+1}m_1$, and $d_{i,3}m_0m_1$, for all $i\in[n-1]$.
\end{enumerate}

As mentioned in the informal discussion, it will be easier to think about the initial infected hypergraph as a set of beachball hypergraphs, and gadgets connecting between them.
For this purpose, we note the following.
\begin{itemize}
    \item The edges from \ref{edges2} and \ref{edges3}, as well as those from \ref{edges2} and \ref{edges4}, (nearly) form beachball hypergraphs.
    For the beachballs with edges from \ref{edges2} and \ref{edges3} the infection process increases with the indices of the middle vertices, whereas for those from \ref{edges2} and \ref{edges4} it decreases with the indices of the middle vertices.
    These hypergraphs are used as the main ingredients of the infection process. 
    \item The second type of edges from \ref{edges5} together with the first type of edges from \ref{edges6} form the gadgets that help us swap from $b_{j}$ to $b_{-j}$, where $t_i$ is fixed; that is, they help us move between the beachball with $t_i,b_j$ as top and bottom and the beachball with $t_i,b_{-j}$.
    \Cref{fig:M2} illustrates the gadget swapping from $b_1$ to $b_{-1}$.
    \item The first type of edges from \ref{edges5} and second type of edges from \ref{edges6} create the gadgets that help us swap from $b_{-j}$ to $b_{j+1}$, where $t_i$ is fixed; that is, these help us move from the beachball with $t_i,b_{-j}$ as top and bottom to the beachball with  $t_i,b_{j+1}$. 
    \item The edges in \ref{edges7} form the gadgets swapping between top vertices, from $t_i$ to $t_{i+1}$, using the dummy vertices $d_{i,s}$.
    In other words, these gadgets move us from the beachball with $t_i,b_{n}$ as top and bottom to the beachball with $t_{i+1},b_1$.
    \Cref{fig:M4} illustrates the gadget swapping from $t_1$ to $t_2$. 
\end{itemize}

We will show that there are three types of edges that are being infected during the process:
\begin{enumerate}[label=(\Roman*)]
    \item missing edges of the beachballs, that is, edges of the form $t_ib_jm_\ell$;
    \item edges from the gadgets swapping bottom vertices, of the form $t_im_{\ell}m_{\ell+1}$, and
    \item edges from the gadgets swapping top vertices (these have several different forms). 
\end{enumerate}

We will now name the edges being infected during the process.
For each $i,j\in [n]$, let $A_{i,j}$ denote the following sequence of edges:
\begin{equation}\label{equation_forwardstage}
A_{i,j}\coloneqq(t_ib_jm_{-(j-2)},\, t_ib_jm_{-(j-3)}, \dots,\, t_ib_jm_{n+j-1},\, t_im_{n+j-1}m_{n+j}).
\end{equation}
These will be the edges infected in the stage of phase $i$ corresponding to the bottom vertex $b_j$.
Similarly, for each $i\in [n]$ and $j\in[n-1]$, let
\begin{equation}\label{equation_backwardstage}
A_{i,-j}\coloneqq(t_ib_{-j}m_{n+j-1},\, t_ib_{-j}m_{n+j-2}, \dots,\, t_ib_{-j}m_{-(j-1)},\, t_im_{-(j-1)}m_{-j}).
\end{equation}
These will correspond to the stage with $b_{-j}$ as bottom vertex.
Concatenating these, we get the sequence $A_i$ of edges corresponding to phase $i$ (these are edges of types (I) and (II) above):
\[A_i\coloneqq A_{i,1}A_{i,-1}A_{i,2}A_{i,-2}\dots A_{i,n-1}A_{i,-(n-1)}A_{i,n}.\]
For the phase change using the dummy vertices $d_{i,j}$, let us write, for each $i\in [n-1]$, 
\begin{equation}\label{equation_dummysequence}
D_i\coloneqq(m_{2n-1}m_{2n}d_{i,1},\, m_{2n}d_{i,1}d_{i,2},\, d_{i,1}d_{i,2}d_{i,3},\, d_{i,2}d_{i,3}t_{i+1},\, d_{i,3}t_{i+1}m_0,\, t_{i+1}m_0m_1).
\end{equation}
These are the edges of Type (III).
Finally, let us write $A$ for the concatenation
\[A\coloneqq A_1D_1A_2D_2\dots A_{n-1}D_{n-1}A_n.\]
We will show that during the infection process, edges become infected one-by-one, according to the sequence $A$.

Let $T$ be the number of triples in $A$, and let $A=(e_1,e_2,\dots,e_T)$. Note that $T=\Theta(n^3)$.
Let us also write $e_0$ for the edge $t_1m_0m_1$, and for all $a\in[T-1]$ let $H_a$ be the copy of $K_4^{(3)}$ with vertex set $e_{a-1}\cup e_a$ (note that $|e_{a-1}\cap e_a|=2$ for all $a$).
For each $s\in[T]$, let us write $G_s$ for the hypergraph with edge set $E(G_0)\cup\{e_1,e_2,\dots,e_s\}$.
(Note that we do not yet know that these coincide with the hypergraphs obtained during the $K_4^{(3)}$-bootstrap process, but we will see that they do.)
Let us also write $G_{-1}\coloneqq G_0-e_0$.

\begin{claim}
Assume that\/ $s\in[-1,T]$ is an integer and\/ $e=x_1x_2x_3$ is a triple not contained in\/ $E(G_s)$.
Suppose that adding\/ $e$ to\/ $G_s$ completes a copy of\/ $K^{(3)}_{4}$ whose fourth vertex is\/ $x_4$.
Then,\/ $s\in[0,T-1]$,\/ $e=e_{s+1}$ and\/ $\{x_1,x_2,x_3,x_4\}=V(H_{s+1})$.
\end{claim}

We note here that the case $s=-1$ is needed to formally justify that $G_0$ is $K_4^{(3)}$-civilised below.

\begin{claimproof}
We consider the following two cases.

	 \textbf{Case 1}: a vertex $d_{i,c}$ appears among $x_1,\dots,x_4$ (for some $i\in[n]$ and $c\in[3]$).
	 Let us temporarily write $d_{i,-2}\coloneqq t_i$, $d_{i,-1}\coloneqq m_{2n-1}$, $d_{i,0}\coloneqq m_{2n}$, $d_{i,4}\coloneqq t_{i+1}$, $d_{i,5}\coloneqq m_0$ and $d_{i,6}\coloneqq m_1$, so the edges $d_{i,a}d_{i,a+1}d_{i,a+3}$ and $d_{i,a}d_{i,a+2}d_{i,a+3}$ are present in $G_0$ for all $-2\leq a\leq 3$.
	 Observe that the only vertices appearing in an edge of $G_s$ together with $d_{i,c}$ (recall $1\leq c\leq 3$) are of the form $d_{i,a}$ with $|c-a|\leq 3$ (see~\ref{edges7} as well as~\eqref{equation_dummysequence}).
	 Hence, $x_1,\dots,x_4$ are all of the form $d_{i,a}$ for some $-2\leq a\leq 6$.
	 Observe furthermore that every edge of $G_s$ of the form $d_{i,p}d_{i,q}d_{i,r}$ ($-2\leq p<q<r\leq 6$) satisfies $|r-p|\leq 3$, or $(p,q,r)=(-2,5,6)$ or $(p,q,r)=(-1,0,4)$.
	 It is easy to deduce that the only possible quadruples of vertices $d_{i,a}$ forming a $K^{(3)}_4$ minus an edge are of the form $\{x_1,x_2,x_3,x_4\}=\{d_{i,a},d_{i,a+1},d_{i,a+2},d_{i,a+3}\}$ (for some $-2\leq a\leq 3$).
	 So exactly one of $d_{i,a}d_{i,a+1}d_{i,a+2}$ and $d_{i,a+1}d_{i,a+2}d_{i,a+3}$ appears in $G_s$, as the other two triples appear in $G_0$ (recall~\ref{edges7}).
	 Since these are the edges $e_N$ and $e_{N+1}$, respectively, for some $N\in[T-1]$, we must have $e=e_{N+1}$ and $e_N\in G_s$.
	 So $N=s$, $e=e_{s+1}$, and $\{x_1,x_2,x_3,x_4\}=\{d_{i,a},d_{i,a+1},d_{i,a+2},d_{i,a+3}\}=e_{s}\cup e_{s+1}=V(H_{s+1})$, as claimed.
	 
	 \textbf{Case 2}: no vertex of the form $d_{i,c}$ ($c\in[3]$) appears among $x_1,\dots,x_4$.
	 Then, $x_1,\dots,x_4$ are all of the form $t_i, b_j$ or $m_\ell$.
	 Observe that no pair of the form $t_it_{i'}$ or $b_jb_{j'}$ appears simultaneously in an edge of $G_s$ ($i\not =i', j\not =j'$), so $X=\{x_1,\dots,x_4\}$ contains at most one vertex of the form $t_i$ and at most one vertex of the form $b_j$.
	 So it must contain at least two vertices of the form $m_\ell $.
	 But $m_\ell $ and $m_{\ell '}$ appear simultaneously in an edge only if $|\ell-\ell'|\leq 1$.
	 It follows that $X$ must be of the form $\{t_i,b_j,m_\ell ,m_{\ell +1}\}$ for some $i,j,\ell$.
	 Assume that $j>0$ (the case $j<0$ is similar).
	 If $\ell\leq -j$, then neither $t_ib_jm_\ell $ nor $b_jm_\ell m_{\ell +1}$ appear in $G_s$ (see~\ref{edges3}, \ref{edges5} and~\eqref{equation_forwardstage}), giving a contradiction.
	 Similarly, if $\ell\geq n+j$, then neither $t_ib_jm_{\ell +1}$ nor $b_jm_\ell m_{\ell +1}$ appear in $G_s$, again giving a contradiction.
	 Hence, we have $-(j-1)\leq \ell\leq n+j-1$.
	 It follows that $b_jm_\ell m_{\ell +1}$ is an edge of $G_0-e_0$.
	 So $e$ is one of $t_ib_jm_\ell $,  $t_ib_jm_{\ell +1}$ and $t_im_\ell m_{\ell +1}$.
	
	First, consider the case $e=t_ib_jm_\ell $.
	Since $t_ib_jm_{\ell +1}$ is already present, we must have $\ell=n+j-1$ (see~\ref{edges5}).
	But if $t_im_\ell m_{\ell +1}=t_im_{n+j-1}m_{n+j}=e_N$ appears in $G_s$, then so does its preceding edge $e_{N-1}=t_ib_jm_{n+j-1}=t_ib_jm_\ell$ (see~\eqref{equation_forwardstage}), giving a contradiction.
	
	If the new edge is $e=t_ib_jm_{\ell +1}$, then $-(j-1)\leq \ell\leq n+j-2$.
	So we have $e=e_N$ for some $N\in[T]$, and the edge $e_{N-1}$ is either $t_ib_jm_\ell $ (if $\ell\not =-(j-1)$) or $t_im_{-(j-2)}m_{-(j-1)}=t_im_\ell m_{\ell +1}$ (if $\ell=-(j-1)$).
	In either case, we have $e_{N-1}\in E(G_s)$ and $e_N\not\in E(G_s)$, giving $s=N-1$, $e=e_{s+1}$, $\{x_1,x_2,x_3,x_4\}=e_{s}\cup e_{s+1}=V(H_{s+1})$, as claimed. 
	
	Finally, consider the case when the new edge is $e=t_im_\ell m_{\ell +1}$. So $t_ib_jm_\ell $ and $t_ib_jm_{\ell +1}$ are edges of $G_s$.
	Note that $t_ib_jm_\ell $ or $t_ib_jm_{\ell +1}$ is of the form $e_N$ for some $N\in[T]$.
	It follows that all edges $e_{N'}$ with $N'<N$ appear in $G_s$, so, in particular, $t_im_{\ell '}m_{\ell '+1}$ is in $G_s$ for all $-(j-1)\leq \ell'\leq n+j-2$.
	Hence $\ell=n+j-1$.
	But then there is some $M\in[T]$ such that  $e_M=t_im_\ell m_{\ell +1}$, and we have $e_{M-1}=t_ib_jm_{n+j-1}=t_ib_jm_{\ell }$, which appears in $G_s$.
	If follows that $s=M-1$, $e=e_{s+1}$ and $\{x_1,x_2,x_3,x_4\}=e_{s}\cup e_{s+1}=V(H_{s+1})$, as claimed.
\end{claimproof}
	
	It is straightforward to check that for all $s\in[T]$ we have $E(H_s)\setminus E(G_{s-1})=\{e_s\}$ and $E(H_s)\cap\{e_0,e_1,\dots,e_s\}=\{e_{s-1},e_s\}$.
	Using these observations and the claim above, we see that all conditions of being $K_4^{(3)}$-civilised are satisfied for $G_0$, and the result follows from Lemma~\ref{lemma_r=3}.
\end{proof}

\begin{remark}\label{rem:bound1}
It immediately follows from the construction and the proof above that our proposed initial infection has\/ $9n+O(1)$ vertices and that the infection process takes\/ $4n^3+O(n^2)$ steps\COMMENT{Fix the top vertex. For the stage of the process with the first bottom vertex we have $n$ steps. In each subsequent stage of the process, the number of steps increases by one. Moreover, we may have some other steps due to the gadgets that allow us to swap top or bottom vertices, but these are smaller order terms. The main order terms thus give us a total of \[n\sum_{i=0}^{2n-1}(n+i)=4n^3\]
steps}.
It therefore follows that\/ $M^3_4(n)\geq4n^3/9^3+O(n^2)$.\COMMENT{One could optimise the construction with some parameters and improve the leading constant to something like 0.006164..., but this is not worth doing since we do not expect our construction to be optimal with respect to the constant.}
We note that we have made no effort to optimise the leading constant.
\end{remark}

\section{Long running times for double infections}\label{sect3}

\subsection{Double infections for $K_4^{(3)}$}
We now move on to the proof of our results about the variant where we allow two edges to be infected at the same time if they together complete a copy of $H$.
We begin with \cref{prop:linear}, giving tight bounds in the case $H=K_4^{(3)}$.
Our approach is motivated by the proof of \citet{BPRS17} of the fact that $M_4(3)=n-3$, but both the construction and the proof of the upper bound are significantly more complicated here.
We start with an informal description of the infection process for the extremal construction.\medskip

We will construct the initially infected hypergraph inductively. Assume that for some $n$ we have already constructed a hypergraph $G_0$ on $n$ vertices $\{x_1,\dots,x_n\}$ for which the process runs for $T$ steps.
Furthermore, assume that $G_T$ is complete, but there exist two vertices $u,v\in \{x_1,\dots,x_n\}$ such that no edge of $G_{T-1}$ contains $u$ and $v$ simultaneously.
(These conditions might at first seem arbitrary, but they are satisfied in the obvious construction when $n=4$.)
Then, we can add another vertex $x_{n+1}$ and another edge $x_{n+1}uv$ to $G_0$ without changing the first $T$ steps of the infection process.
(Indeed, the process will only be altered if we create a new $2$-completable copy of $K_4^{(3)}$, and this requires having two edges sharing two vertices.)
Then, at time $T+1$, the new edges that become infected are all those of the form $x_{n+1}ws$ with $w\in W_1\coloneqq \{u,v\}$ and $s\in \{x_1,\dots,x_n\}\setminus W_1$.
Moreover, at time $T+2$, $x_1,\dots,x_{n+1}$ will form a complete hypergraph.
This gives a contruction on $n+1$ vertices with running time $T+2$.

To obtain our construction for $n+2$ vertices, notice that if we pick some $u_2\in \{x_1,\dots,x_n\}\setminus W_1$, then no edge of $G_{T}$ contains both $x_{n+1}$ and $u_2$.
So if we add a new vertex $x_{n+2}$ and a new edge $x_{n+2}x_{n+1}u_2$, then the first $T+1$ steps of the infection will remain unaffected by this change.
Furthermore, one can check that at time $T+2$ the edges containing $x_{n+2}$ are given as $x_{n+2}x_{n+1}w$ and $x_{n+2}u_2w$ with $w\in W_1$.
Moreover, at time $T+3$ the edges $x_{n+2}ws$ with $w\in W_2\coloneqq W_1\cup\{u_2,x_{n+1}\}$ and $s\in\{x_1,\dots,x_{n+1}\}\setminus W_2$ will become infected, and at time $T+4$ we get a complete hypergraph.

We can keep repeating these steps: take some $u_{j}\in\{x_1,\ldots,x_n\}\setminus W_{j-1}$, add a new vertex $x_{n+j}$ and a new edge $x_{n+j}x_{n+j-1}u_{j}$.
This will extend the process by $2$ steps, and at time $T+2j-1$ the edges containing the new vertex $x_{n+j}$ will be of the form $x_{n+j}ws$ with $w\in W_{j}\coloneqq W_{j-1}\cup\{u_{j},x_{n+j-1}\}$ and $s\in \{x_1,\dots,x_{n+j-1}\}$.
Moreover, at time $T+2j$ our hypergraph will contain all edges on $\{x_1,\dots,x_{n+j}\}$.
This means that we can keep adding a vertex and extending the process by $2$ steps each time.
However, the set $W_j$ is growing, and at some point it will contain all of our vertices.
When this happens, we will no longer be able to pick an appropriate $u_{j}$, and we will `lose' 1 step of the infection process (i.e., by adding a new vertex we can only extend the process by $1$ step at this point).
So `usually' adding a vertex extends the infection by $2$ steps, giving the leading term $2n$ for the running time, but sometimes (when $W$ becomes everything) we only gain one extra step, and this will contribute the term $-\lfloor\log_2(n-2)\rfloor$.
\medskip

Let us now start the formal construction. Let $G_0$ be any $3$-uniform hypergraph on some vertex set $V$, and consider the corresponding $(K_4^{(3)},2)$-process $G_0,G_1,\ldots$
Let $T$ be the running time of this process.
We say that $G_0$ is \emph{nice} if $T\not=0$, $G_T$ is complete, and there exist distinct vertices $u,v\in V$ such that no edge of $G_{T-1}$ contains both $u$ and $v$.
The following lemma will be used to obtain the lower bound.

\begin{lemma}\label{lemma_doubleinfectionconstruction}
    Suppose that there is a nice hypergraph on\/ $k\geq 4$ vertices such that the corresponding\/ $(K_4^{(3)},2)$-process has running time\/ $T$.
    Then, for all\/ $\ell\in[k+1,2k-3]$ we have
    \[M_{(K_4^{(3)},2)}(\ell)\geq T+2(\ell-k).\]
    Furthermore, 
    \[M_{(K_4^{(3)},2)}(2k-2)\geq T+2k-5,\]
    and there exists a nice hypergraph on\/ $2k-2$ vertices whose corresponding\/ $(K_4^{(3)},2)$-process has running time\/ $T+2k-5$.
\end{lemma}

\begin{proof}
Let $G_0$ be a nice $3$-uniform hypergraph on $k$ vertices $x_1,\dots,x_k$ such that the corresponding $(K_4^{(3)},2)$-process $G_0,G_1,\dots$ has running time $T$, $G_T$ is complete, and $x_{1},x_k$ do not appear in any edge of $G_{T-1}$ simultaneously.
Let $\ell\in[k+1,2k-2]$ be arbitrary.
We define a hypergraph $G_0'$ on a vertex set $\{x_1,\dots,x_\ell\}$ of size $\ell$ as follows.
For any $i\in[\ell-k]$, let 
\[e_i\coloneqq x_{k+i}x_{k+i-1}x_{i},\]
and let $\mathcal{E}\coloneqq\{e_i:i\in[\ell-k]\}$.
Then, set
\[E(G_0')\coloneqq E(G_0)\cup\mathcal{E}.\]
Let $G_0',G_1',\dots$ be the corresponding $(K_4^{(3)},2)$-bootstrap percolation process with initial infection $G_0'$, and let us write $W_j\coloneqq\{x_1,\dots,x_j,x_k,x_{k+1},\dots,x_{k+j-1}\}$ for all $j\in [\ell-k]$.

\begin{claim}
We have
\begin{equation*}
    E(G_t')=
    \begin{cases}
    E(G_t)\cup\mathcal{E}&\text{ if\/ } t\in[T],\\\vspace{0.2em}
    \displaystyle\binom{\{x_1,\dots,x_{k+j-1}\}}{3}\cup\mathcal{E}\cup F_j^{\mathrm{odd}}&\text{ if\/ }t=T+2j-1\text{ with }j\in[\ell-k],\\\vspace{0.5em}
    \displaystyle\binom{\{x_1,\dots,x_{k+j}\}}{3}\cup\mathcal{E}\cup F_j^{\mathrm{even}}&\text{ if\/ }t=T+2j\text{ with }j\in[\ell-k],
    \end{cases}
\end{equation*}
where
\[F_j^\mathrm{odd}\coloneqq\{x_{k+j}wx_a:w\in W_j, a\in[k+j-1], x_a\neq w\}\]
and
\[F_j^\mathrm{even}\coloneqq\{x_{k+j+1}wx_a: w\in W_j,a\in\{j+1,k+j\}\},\]
unless\/ $j=\ell-k$, in which case\/ $F_{\ell-k}^\mathrm{even}\coloneqq\varnothing$.
\end{claim}

\begin{claimproof}
We show this statement by induction on $t$.
The case $t\in[T]$ is straightforward.
Indeed, note that $|e_i\cap e_j|< 2$ for all distinct $i,j\in[\ell-k]$, and $|e_i\cap f|<2$ whenever $f\in E(G_{T-1})$ (by our assumption that $G_{T-1}$ does not contain any triple containing both $x_1$ and $x_{k}$).
Recall that, if a copy $H$ of $K_4^{(3)}$ is $2$-completable in a hypergraph $G$, then $G$ contains two edges of $H$, which must share two vertices.
Thus, in the first $T$ steps of the process, the addition of the edges in $\mathcal{E}$ does not result in any infections that did not occur for $G_0$.
Now assume that $t>T$ and the statement above holds for $t-1$.
For notational purposes, set $F_0^\mathrm{even}\coloneqq\varnothing$.
We split the analysis into two cases.

\textbf{Case 1.} 
Consider first the case $t=T+2j-1$ (with $j\in[\ell-k]$). 
By the induction hypothesis,
\[E(G_{t-1}')=\binom{\{x_1,\dots,x_{k+j-1}\}}{3}\cup\mathcal{E}\cup F_{j-1}^{\mathrm{even}}.\]

We first verify that $F_j^{\mathrm{odd}}\subseteq E(G_t')$.
Let $w\in W_j$ and $a\in[k+j-1]$ with $x_a\neq w$, so that $x_{k+j}wx_a\in F_j^{\mathrm{odd}}$.
Then, there exists some $w'\in W_j$ with $x_{k+j}ww'\in E(G_{t-1}')$.
(Indeed, if we pick $w'$ such that $w'\in \{x_j,x_{k+j-1}\}$ and $w'\not =w$, then $x_{k+j}ww'\in F_{j-1}^{\mathrm{even}}\cup\{e_{j}\}$.)
If $x_a=w'$, then trivially $x_{k+j}wx_a\in E(G_t')$.
If $x_a\neq w'$, then we also have $ww'x_a\in E(G_{t-1}')$.
It follows that the copy of $K_4^{(3)}$ with vertex set $\{x_{k+j},w,w',x_a\}$ is $2$-completable in $G_{t-1}'$, and so $x_{k+j}wx_a\in E(G_t')$.

We next show that any edge infected at time $t$ must belong to $F_j^{\mathrm{odd}}$.
Indeed, if $h\in[j+1,\ell-k]$ then $|e_h\cap f|<2$ for all $f\in E(G_{t-1}')\setminus\{e_h\}$, so $e_h$ cannot appear in a copy of $K_4^{(3)}$ completed in this step.
It follows that any added edge must be of the form $e=x_{k+j}x_ax_b$ with $a,b\in[k+j-1]$ distinct.
Furthermore, either $x_a$ or $x_b$ must appear together with $x_{k+j}$ in an edge of $E(G_{t-1}')\setminus\{e_i:i\in[j+1,\ell-k]\}$, so one of $x_a,x_b$ must belong to $W_{j-1}\cup\{x_j,x_{k+j-1}\}=W_j$.
But then $e\in F_j^\mathrm{odd}$, as claimed.

\textbf{Case 2.}
Consider now the case $t=T+2j$ (with $j\in[\ell-k]$).
By induction, we know
\[E(G_{t-1}')=\binom{\{x_1,\dots,x_{k+j-1}\}}{3}\cup\mathcal{E}\cup F_{j}^{\mathrm{odd}}.\]

Observe first that, whenever $a\in[k+j-2]$, we have $x_{k+j}x_{k+j-1}x_a\in F_{j}^{\mathrm{odd}}\subseteq E(G_{t-1}')$.
It follows that, whenever $a,b\in [k+j-2]$ are distinct, the copy of $K_4^{(3)}$ with vertex set $\{x_{k+j},x_{k+j-1},x_a,x_b\}$ is $2$-completable in $G_{t-1}'$.
Hence, $x_{k+j}x_{c}x_{d}\in E(G_t')$ whenever $c,d\in[k+j-1]$ (distinct).
Thus, $\inbinom{\{x_1,\dots,x_{k+j}\}}{3}\subseteq E(G_t')$. 
Furthermore, assume that $j\not =\ell-k$ and $e\in F_j^\mathrm{even}$, so $e=x_{k+j+1}wx_a$ with $w\in W_j$ and $a\in\{j+1,k+j\}$.
Then, $e_{j+1}=x_{k+j+1}x_{k+j}x_{j+1}\in E(G_{t-1}')$ and $x_{k+j}wx_{j+1}\in F_{j}^{\mathrm{odd}}\subseteq E(G_{t-1}')$, so the copy of $K_4^{(3)}$ with vertex set $\{x_{k+j+1},x_{k+j},x_{j+1},w\}$ is $2$-completable in $G_{t-1}'$, which implies $x_{k+j+1}wx_a\in E(G_t')$.
So $F_j^\mathrm{even}\subseteq E(G_t')$.

It remains to show that any edge added in this step must belong to $\inbinom{\{x_1,\dots,x_{k+j}\}}{3}\cup F_j^{\mathrm{even}}$.
Indeed, as in the previous case, we see that any copy of $K_4^{(3)}$ which is $2$-completable in $G_{t-1}'$ must have vertex set $\{x_a,x_b,x_c,x_d\}$ with $a,b,c,d\in[k+j+1]$ (distinct).
So any edge which is infected at time $t$ is of the form $e=x_ax_bx_c$ with $a,b,c\in[k+j+1]$.
If $a,b,c\in[k+j]$, the containment holds trivially, so we may assume that $c=k+j+1$.
In order for $e$ to become infected at time $t$, we must have that $x_{k+j+1}x_ax_d$ or $x_{k+j+1}x_bx_d$ appears in $E(G_{t-1}')$; we may assume that it is the former.
But this edge must be $e_{j+1}=x_{k+j+1}x_{k+j}x_{j+1}$,\COMMENT{This follows from the induction hypothesis: it is straightforward to check that this is the only edge in $G_{t-1}'$ whose largest vertex has label at most $k+j+1$ and which contains a vertex with this label.} and hence $\{a,d\}=\{j+1,k+j\}$.
It also follows that $x_{k+j+1}x_bx_d\not \in E(G_{t-1}')$, and hence $x_ax_bx_d\in E(G_{t-1}')$, i.e., $x_{k+j}x_bx_{j+1}\in E(G_{t-1}')$. 
This implies $x_{k+j}x_bx_{j+1}\in F_{j}^\mathrm{odd}$\COMMENT{It is trivial that this edge does not lie in $\inbinom{\{x_1,\dots,x_{k+j-1}\}}{3}$, and straightforward to check it cannot be in $\mathcal{E}$ because we have $b\neq c$.}
and, therefore, $x_b\in W_j$.
So $c=k+j+1$, $a\in\{j+1,k+j\}$ and $x_b\in W_j$, hence $x_ax_bx_c\in F_j^\mathrm{even}$, as claimed.
\end{claimproof}

By the claim above, $G_{T+2(\ell-k)}'$ is complete, but $G_{T+2(\ell-k)-1}'$ is not unless $\ell=2k-2$ (indeed, if $\ell\neq2k-2$, then $x_{k-2}x_{k-1}x_\ell\notin E(G_{T+2(\ell-k)-1}')$). 
Moreover, if $\ell=2k-2$, then $G_{T+2(\ell-k-1)}'=G_{T+2k-6}'$ does not contain an edge in which both $x_{2k-2}$ and $x_{k-1}$ appear.
The statement of the lemma follows.
\end{proof}

We are ready to deduce the lower bound.

\begin{lemma}\label{lemma_doubleinfectionlowerbound}
For all\/ $n\geq 4$ we have\/ $M_{(K_{4}^{(3)},2)}(n)\geq2n-\lfloor\log_2(n-2)\rfloor-6$.
\end{lemma}

\begin{proof}
Observe first that there is a nice $3$-uniform hypergraph $G_0$ on $4$ vertices $\{x_1,x_2,x_3,x_4\}$, given by $E(G_0)=\{x_1x_2x_3,x_2x_3x_4\}$, for which the running time of the $(K_{4}^{(3)},2)$-process is $T=1$. 
A straightforward induction using \cref{lemma_doubleinfectionconstruction} shows that for all $m\geq 1$ there exists a nice hypergraph on $2^m+2$ vertices for which the running time of the $(K_{4}^{(3)},2)$-process is $2^{m+1}-(m+2)$.
Furthermore, also by \cref{lemma_doubleinfectionconstruction}, whenever $2^m+2\leq n< 2^{m+1}+2$ we have
\[M_{(K_{4}^{(3)},2)}(n)\geq 2^{m+1}-(m+2)+2(n-2^m-2)=2n-m-6=2n-\lfloor\log_2(n-2)\rfloor-6,\]
as claimed.
\end{proof}

We now turn to the proof of the upper bound. 
For any $t\geq 1$, let $m:=m(t)$ denote the unique positive integer  satisfies
\[2^{m+1}-(m+2)\leq t<2^{m+2}-(m+3).\]
The following key lemma essentially shows that the infections must contain a substructure similar to the one in our construction.

\begin{lemma}\label{lemma_doubleinfectionK4}
	Let\/ $G_0$ be a\/ $3$-uniform hypergraph on\/ $n\geq4$ vertices, and consider the\/ $(K_4^{(3)},2)$-process\/ $G_0,G_1,\dots$ with\/ $G_0$ as initial infection.
	Assume that\/ $a\geq 1$ and\/ $e\in E(G_a)\setminus E(G_{a-1})$.
	Then, there exist some\/ $t_e\geq a$,\/ $S_e\subseteq V(G_0)$,\/ $v_e\in S_e$ and\/ $W_e\subseteq S_e\setminus\{v_e\}$ such that, for\/ $m=m(t_e)$,
	\begin{enumerate}[label=$(\mathrm{P}\arabic*)$]
		\item\label{prop2} $e\subseteq S_e$,
		\item\label{prop3} $|S_e|=({t_e+m+6})/{2}$,
		\item\label{prop4} $|W_e|=t_e-(2^{m+1}-(m+2))$,
		\item\label{prop1} $G_{t_e}[S_e]$ is complete, and
		\item\label{prop5} $\inbinom{S_e\setminus\{v_e\}}{3}\cup\{v_ews:w\in W_e,s\in S_e\setminus\{v_e,w\}\}\subseteq E(G_{t_e-1})$.
	\end{enumerate}
\end{lemma}

\begin{proof}
	We prove the statement by induction on $a$.
	If $a=1$, then we know $e$ is in some copy $H$ of $K_4^{(3)}$ in $G_1$.
	We can set $t_e=1$ (so $m=1$), $S_e=V(H)$, $v_e\in V(H)$ such that $V(H)\setminus\{v_e\}\in E(G_0)$, and $W_e=\varnothing$; the properties \ref{prop2}--\ref{prop5} are then satisfied.
	
	Now assume that $a\geq 2$ and the statement holds for smaller values of $a$.
	We know there is some copy $H$ of $K_4^{(3)}$ in $G_a$ such that $e\in E(H)$, $|E(H)\cap E(G_{a-1})|\geq 2$ and $|E(H)\cap E(G_{a-2})|< 2$.
	It follows that there is some $f\in E(H)$ such that $f\in E(G_{a-1})\setminus E(G_{a-2})$, i.e., $f$ is infected at time $a-1$.
	Furthermore, there is another edge $f'\in E(H)\cap E(G_{a-1})$ (with $f'\not =f$).
	Let us write $t\coloneqq t_f$, $S\coloneqq S_f$, $W\coloneqq W_f$ and $v\coloneqq v_f$, and let $m=m(t)$.
	We consider several cases according to how $e$, $f$ and $f'$ overlap with $S$ (and $v$, $W$).
	Note that, if $e\not \subseteq S$, then there is some $p\in V(G_0)$ such that $e\setminus f=e\setminus S=\{p\}$ and $p\in f'$.
	
	\textbf{Case 1:} $e\subseteq S$.
	Since $e\not \in E(G_{a-1})$ and $G_t[S]$ is complete, we have $t\geq a$.
	It follows that $t_e=t$, $S_e=S$, $v_e=v$, $W_e=W$ satisfy properties \ref{prop2}--\ref{prop5}.
	
	\textbf{Case 2:} $e\not \subseteq S$, $f'\in E(G_{t-1})$, and $f'=pss'$ for some $s,s'\in S\setminus W$ (recall $\{p\}=e\setminus S$).
	By \ref{prop5} for $f$, we know that, whenever $w\in W$, we have $wss'\in E(G_{t-1})$.\COMMENT{Note that both this and the next sentence are vacuously true if $W=\varnothing$.}
	This, together with the fact that $f'=pss'\in E(G_{t-1})$, guarantees that 
	\begin{equation}\label{eq:upper1}
	    pws, pws'\in E(G_t). 
	\end{equation}
	Let $u\in W\cup\{s,s'\}$ and $u'\in \{s,s'\}$ with $u'\not =u$, and let $z\in S\setminus\{s,s'\}$ with $z\not =u$.\COMMENT{\ref{prop3} and \ref{prop4} guarantee that these choices always exist.
	Indeed, using the upper bound in the definition of $m$,
	\[|S\setminus W|=\frac{t+m+6}{2}-(t-(2^{m+1}-(m+2)))=2^{m+1}-\frac{t+m}{2}+1>2^{m+1}-\frac{2^{m+2}-(m+3)+m}{2}+1=\frac{5}{2}.\]}
	Then, \eqref{eq:upper1} and our assumption on $f'$ tell us that $puu'\in E(G_t)$, and by \ref{prop1} for $f$ we have that $uu'z\in E(G_t)$.
	This implies $puz\in E(G_{t+1})$.
	That is, for all $u\in W\cup\{s,s'\}$ and all $z\in S\setminus \{u\}$ we have $puz\in E(G_{t+1})$.
	This in turn implies that $G_{t+2}[S\cup\{p\}]$ is complete (as $psz,psz'\in E(G_{t+1})$ implies $pzz'\in E(G_{t+2})$).
	
	If $t=2^{m+2}-(m+4)$, then $|W|=2^{m+1}-2$ and $|S|=2^{m+1}+1$, so we have $|W\cup\{s,s'\}|=|S|-1$ and hence $G_{t+1}[S\cup\{p\}]$ is complete by the observation above.
	Hence, $t_e=t+1$, $S_e=S\cup\{p\}$, $v_e=p$, $W_e=\varnothing$ satisfy properties \ref{prop2}--\ref{prop5} (note that in this case $m(t_e)=m+1$).
	
	On the other hand, if $t\not =2^{m+2}-(m+4)$, then $t\leq 2^{m+2}-(m+6)$ (since $t+m=2|S|-6$ is even by \ref{prop3}). So $m(t+2)=m(t)$.
	It follows that $t_e=t+2$, $S_e=S\cup\{p\}$, $v_e=p$, $W_e=W\cup\{s,s'\}$ satisfy the properties.
	
	\textbf{Case 3:} $e\not \subseteq S$, $f'\in E(G_{t-1})$, and $f'=pws$ for some $w\in W$, $s\in S$, $\{p\}=e\setminus S$.
	Then, by \ref{prop5} for $f$, whenever $z\in S\setminus \{w,s\}$ we have $wsz\in E(G_{t-1})$.
	Since $f'=pws\in E(G_{t-1})$, it follows that $pwz\in E(G_t)$ for all $z\in S\setminus\{w\}$.
	Therefore, whenever $z,z'\in S\setminus\{w\}$ are distinct, we have $pwz,pwz'\in E(G_t)$, and hence $pzz'\in E(G_{t+1})$.
	Thus, $G_{t+1}[S\cup\{p\}]$ is complete. 
	
	If $t=2^{m+2}-(m+4)$, then $t_e=t+1$, $S_e=S\cup\{p\}$, $v_e=p$, $W_e=\varnothing$ satisfy properties \ref{prop2}--\ref{prop5}.
	
	On the other hand, assume that $t<2^{m+2}-(m+4)$ (as in case 2, we then have $m(t+2)=m(t)$). 
	Then, \ref{prop3} and \ref{prop4} imply that $|W|< |S|-3$.
	Let $W'$ be an arbitrary subset of $S$ of size $|W|+2$.
	Then, $t_e=t+2$, $S_e=S\cup\{p\}$, $v_e=p$, $W_e=W'$ satisfy properties \ref{prop2}--\ref{prop5}.
	
	\textbf{Case 4:} $e\not \subseteq S$, $f'\not \in E(G_{t-1})$.
	Then, we must have $t=a-1$ and $f'\in E(G_t)\setminus E(G_{t-1})$. 
	We may assume that $t_{f'}=t$, since otherwise we can swap the roles of $f$ and $f'$ and we are done by the previous cases.
	Let us write $S'$ for $S_{f'}$.
	Note that $|S|=|S'|$, and $S\cap S'\supseteq f\cap f'$ has size at least $2$.
	
	Assume first that $S'\setminus S=\{p\}$ for some $p\in V(G_0)$ (where necessarily $\{p\}=e\setminus f$). 
	Then, $S\setminus S'=\{q\}$ for some $q\in V(G_0)$ (with $\{q\}=e\setminus f'$).
	Observe that, by \ref{prop1}, whenever $s,s'\in S\cap S'$ are distinct, we have $pss'\in E(G_t)$ and $qss'\in E(G_t)$. This implies that $pqs\in E(G_{t+1})$ for every $s\in S\cap S'$.
	Hence, $G_{t+1}[S\cup S']$ is complete, where $|S\cup S'|=|S|+1$.
	If $t=2^{m+2}-(m+4)$, then properties \ref{prop2}--\ref{prop5} are satisfied for $t_e=t+1$, $S_e=S\cup S'$, $v_e=p$ and $W_e=\varnothing$.
	Otherwise, $t\leq 2^{m+2}-(m+6)$ (as $t+m$ is even), so \ref{prop2}--\ref{prop5} are satisfied for $t_e=t+2$, $S_e=S\cup S'$, $v_e=p$, and $W_e$ an arbitrary subset of $S$ of size $|W|+2$.
	
	Now assume that $|S'\setminus S|\geq 2$.
	Observe that, whenever $x,y\in S\cap S'$ (distinct), $s\in S\setminus S'$ and $s'\in S'\setminus S$, by \ref{prop1} we know $xys, xys'\in E(G_t)$, and therefore $xss'\in E(G_{t+1})$.
	If $s\in S\setminus S'$ and $s',\bar{s}'\in S'\setminus S$ are distinct, then $xs\bar{s}'\in E(G_{t+1})$ (by \ref{prop1}).
	Together with the fact that $xss'\in E(G_{t+1})$, we conclude that $ss'\bar{s}'\in E(G_{t+2})$.
	Similarly, if $s'\in S'\setminus S$ and $s,\bar{s}\in S\setminus S'$ (distinct), then $s\bar{s}s'\in E(G_{t+2})$.
	Hence, $G_{t+2}[S\cup S']$ is complete, where $|S\cup S'|\geq|S|+2$.
	Now pick any two vertices $p,p'\in S'\setminus S$ with $e\subseteq S\cup\{p,p'\}$.
	If $t=2^{m+2}-(m+6)$, then let $t_e=t+3$, $S_e=S\cup\{p,p'\}$, $v_e=p$ and $W_e=\varnothing$.
	If $t=2^{m+2}-(m+4)$, then let $t_e=t+3$, $S_e=S\cup\{p,p'\}$, $v_e=p$ and $W_e$ an arbitrary subset of $S\cap S'$ of size $2$.
	Finally, assume $t<2^{m+2}-(m+6)$.
	Since $t+m$ is even, we have $t\leq 2^{m+2}-(m+8)$.
	Then, let $t_e=t+4$, $S_e=S\cup\{p,p'\}$, $v_e=p$ and $W_e$ an arbitrary subset of $S$ of size $|W|+4$.
	These choices satisfy properties \ref{prop2}--\ref{prop5}.
	This finishes the proof of the lemma.	
\end{proof}

\begin{proof}[Proof of \cref{prop:linear}]
The lower bound follows from \cref{lemma_doubleinfectionlowerbound}.
For the upper bound, given an arbitrary hypergraph $G_0$ on a vertex set $V$ of size $n\geq 4$, consider the $(K_4^{(3)},2)$-process $G_0,G_1,\dots$ with $G_0$ as initial infection. 
Let $T\coloneqq M_{(K_4^{(3)},2)}(G_0)$ be the running time of the process, and let $e\in E(G_T)\setminus E(G_{T-1})$ be arbitrary.
By \cref{lemma_doubleinfectionK4}, there is some $t\geq T$ such that $G_t$ contains a clique of size $\frac{t+m(t)+6}{2}\geq \frac{T+m(T)+6}{2}$.
Hence,
\[\frac{T+m(T)+6}{2}\leq n.\]
It follows that
\[T\leq 2n-\lfloor\log_2(n-2)\rfloor-6,\]
as we wanted to prove.
Indeed, if $\lfloor\log_2(n-2)\rfloor=\alpha$, then $2^\alpha+2\leq n\leq 2^{\alpha+1}+1$ and hence $2^{\alpha+1}-(\alpha+2)\leq 2n-\alpha-6< 2^{\alpha+2}-(\alpha+3)$, giving $m(2n-\alpha-6)=\alpha$ and $\frac{(2n-\alpha-6)+m(2n-\alpha-6)+6}{2}=n$.
\end{proof}

\subsection{Double infections for $K_5^{(3)}$}
We now turn our attention to \cref{thm:cubic}.
In order to prove it, observe that, for an $r$-uniform hypergraph $H$, the trivial upper bound $M_{(H,k)}(n)\leq\binom{n}{r}$ still holds, so it suffices to provide a lower bound.
We will proceed by constructing an initial infection for which the $(K_{5}^{(3)},2)$-bootstrap percolation process runs for a cubic number of steps.
At most two edges will become infected in each step of the infection process, which will make it easier to analyse the number of steps.
Our construction is intuitively similar to the one we constructed for the proof of \cref{theorem_bootstrapgeneral}, albeit a bit more convoluted.
Let us begin with an intuitive description.

Our vertex set will again be split into three layers, with vertices $t_i$ playing the role of `top' vertices, vertices $b_j$ playing the role of `bottom' vertices, and vertices $m_\ell$ conforming the `middle' layer.
We will also have a number of `dummy' vertices.
For fixed top and bottom vertices, the vertices in the middle layer will allow us to infect two edges at a time, while traversing this layer, for a linear total number of steps.
For each fixed bottom vertex, there will be some extra edges at the end of the middle layer which will allow us to swap the bottom vertex for the next one and continue the process.
Finally, the dummy vertices will allow us to swap the top vertex and start the process anew.

To be more precise, let us describe the first few stages of the infection process.
For each $0\leq\ell\leq n$, our initial infection will contain all edges of the copy of $K_5^{(3)}$ with vertex set $t_1,b_1,m_\ell,m_{\ell+1},m_{\ell+2}$ except for $t_1m_{\ell}b_1$, $t_1m_{\ell+1}b_1$, and $t_1m_{\ell+2}b_1$.
It will also contain $t_1m_0b_1$.
This edge will trigger the infection of $t_1m_1b_1$ and $t_1m_2b_1$ in the first step of the process, then of $t_1m_3b_1$ and $t_1m_4b_1$, and the infection will keep propagating towards higher values of $\ell$, until finally, after $n$ steps, the edges $t_1m_{2n-1}b_1$ and $t_1m_{2n}b_1$ become infected.

At this point, we will swap out $b_1$ to $b_{-1}$.
This can be achieved in two steps of the infection process.
Our initial infection will already contain all edges of the copy of $K_5^{(3)}$ with vertex set $t_1,b_1,m_{2n},m_{2n+1},m_{2n+2}$ except for $t_1m_{2n}m_{2n+1}$, $t_1m_{2n+1}m_{2n+2}$, and the edge $t_1m_{2n}b_1$ which was just added in the previous step; $t_1m_{2n}m_{2n+1}$ and $t_1m_{2n+1}m_{2n+2}$ will therefore become infected in the next step.
The initial infection also contains all the edges of the copy of $K_5^{(3)}$ with vertex set $t_1,b_{-1},m_{2n},m_{2n+1},m_{2n+2}$ except for $t_1m_{2n}b_{-1}$, $t_1m_{2n+1}b_{-1}$, and the two that were just added.
These two edges now become infected as well, and start a new infection process where now the indices decrease through the middle layer.

Finally, suppose we have reached a point where the copy of $K_5^{(3)}$ defined on the vertices $t_1$, $b_n$, $m_{4n-4}$, $m_{4n-3}$, $m_{4n-2}$ has been completely infected, with the edges infected in the last step of the process being $t_1m_{4n-3}b_n$ and $t_1m_{4n-2}b_n$.
We now want to swap out the top vertex to $t_2$, using for this purpose four dummy vertices.
Similarly as in the proof of \cref{theorem_bootstrapgeneral}, these will simply create a short chain of infections that allows us to restart the process.

We now give a formal proof.

\begin{proof}[Proof of \cref{thm:cubic}]
Consider an initial infection hypergraph $G_0$ whose vertex set consists of $12n-5$ vertices, labelled as $t_1,\ldots,t_n$, $b_1,\ldots,b_n$, $b_{-1},\ldots,b_{-(n-1)}$, $m_{-2(n-1)},\ldots,m_{4n}$, and $d_{i,1},d_{i,2},d_{i,3}$ for $i\in[n-1]$.
For notational purposes, for each $i\in[n]$ let $d_{i,-2}\coloneqq t_i$, $d_{i,-1}\coloneqq b_n$, $d_{i,0}\coloneqq m_{4n-2}$, $d_{i,4}\coloneqq m_0$, $d_{i,5}\coloneqq b_1$, and $d_{i,6}\coloneqq t_{i+1}$.
The edges of $G_0$ appear in the following list:
\begin{enumerate}[label=(\alph*)]
    \item\label{initial1} $t_1m_0b_1$;
    
    \item\label{initial2} $m_\ell m_{\ell+1}m_{\ell+2}$, for all $-2(n-1)\leq\ell\leq4n-4$;
    
    \item\label{initial3} $t_im_\ell m_{\ell+2}$, for all $i\in[n]$ and $-2(n-1)\leq\ell\leq4n-4$;\COMMENT{We do not actually need all of these edges, but not including the ones that we do not explicitly use would make some steps of the process include three edges at once, which we have said we want to avoid.}
    \item\label{initial4} $b_jm_\ell m_{\ell+2}$, for all $j\in[n]$ and $-2(j-1)\leq\ell\leq2(n+j-1)$;
    \item\label{initial5} $b_{-j}m_\ell m_{\ell+2}$, for all $j\in[n-1]$ and $-2j\leq\ell\leq2(n+j-1)$;
    
    \item\label{initial6} $t_im_\ell m_{\ell+1}$, for all $i\in[n]$ and $0\leq\ell\leq 2n-1$;
    \item\label{initial7} $b_jm_\ell m_{\ell+1}$, for all $j\in[n]$ and $-2(j-1)\leq\ell\leq2(n+j)-1$;
    \item\label{initial8} $b_{-j}m_\ell m_{\ell+1}$, for all $j\in[n-1]$ and $-2j\leq\ell\leq2(n+j)-1$;
    
    \item\label{initial9} $t_im_{2(n+j)-1}b_j$, $t_im_{2(n+j)}b_j$ and $t_im_{2(n+j)}b_{-j}$, for all $i\in[n]$ and $j\in[n-1]$;
    \item\label{initial10} $t_im_{-2j+1}b_{-j}$, $t_im_{-2j}b_{-j}$ and $t_im_{-2j}b_{j+1}$, for all $i\in[n]$ and $j\in[n-1]$;
    
    \item\label{initial11} $d_{i,j}d_{i,j+1}d_{i,j+3}$, $d_{i,j}d_{i,j+1}d_{i,j+4}$, $d_{i,j}d_{i,j+2}d_{i,j+3}$, $d_{i,j}d_{i,j+2}d_{i,j+4}$, $d_{i,j}d_{i,j+3}d_{i,j+4}$, $d_{i,j+1}d_{i,j+2}d_{i,j+3}$,\linebreak{} and $d_{i,j+1}d_{i,j+3}d_{i,j+4}$, for all $i\in[n-1]$ and $j\in\{-2,0,2\}$.
\end{enumerate}
To compare this with the construction hinted at before the proof, consider the following.
The edge in \ref{initial1} is an edge $e_0$ which starts the whole infection process.
The edges in \ref{initial2}--\ref{initial8} are there to ensure the correct propagation of the infection through the middle layer, where the edges in \ref{initial4} and \ref{initial7} will be used to propagate the infection towards larger values of $\ell$, using some bottom vertex of the form $b_j$, and those in \ref{initial5} and \ref{initial8} will be used to propagate the infection towards smaller values of $\ell$, using some bottom vertex of the form $b_{-j}$.
The edges in \ref{initial9} and \ref{initial10} are needed for swapping the bottom vertices.
Finally, the edges which appear in \ref{initial11} are used to swap the top vertices.

For each pair $(i,j)$ with $i,j\in[n]$, let $A_{i,j}$ be the sequence of edges
\begin{equation}\label{eq:2perc1}
    A_{i,j}\coloneqq(t_im_{-2(j-1)+\ell}b_j)_{\ell=1}^{2n+4(j-1)}.
\end{equation}
Similarly, for each pair $(i,j)$ with $i\in[n]$ and $j\in[n-1]$, we define
\begin{equation}\label{eq:2perc2}
    A_{i,-j}\coloneqq(t_im_{2(n+j)-\ell}b_{-j})_{\ell=1}^{2n+4(j-1)+2}.
\end{equation}
Note that each of these has an even number of elements.
Using $\bigtimes$ to denote concatenation, for each phase $i\in[n]$ we define the sequence
\begin{equation}\label{eq:2perc3}
    A_i\coloneqq A_{i,1}\bigtimes_{j=1}^{n-1}(t_im_{2(n+j-1)}m_{2(n+j)-1},t_im_{2(n+j)-1}m_{2(n+j)})A_{i,-j}(t_im_{-2(j-1)}m_{-2j+1},t_im_{-2j+1}m_{-2j})A_{i,j+1}.
\end{equation}
Finally, we set
\begin{align}\label{eq:2perc4}
    A\coloneqq\ & A_1\bigtimes_{i=1}^{n-1}(d_{i,-1}d_{i,0}d_{i,2},d_{i,0}d_{i,1}d_{i,2},d_{i,1}d_{i,2}d_{i,4},d_{i,2}d_{i,3}d_{i,4},d_{i,3}d_{i,4}d_{i,6},d_{i,4}d_{i,5}d_{i,6})A_{i+1}\nonumber\\
    =\ &A_1\bigtimes_{i=1}^{n-1}(b_nm_{4n-2}d_{i,2},m_{4n-2}d_{i,1}d_{i,2},d_{i,1}d_{i,2}m_0,d_{i,2}d_{i,3}m_0,d_{i,3}m_0t_{i+1},m_0b_1t_{i+1})A_{i+1}.
\end{align}
We will sometimes abuse notation and treat each of the above sequences as sets.
Observe that $A$ has an even number of elements and that none appear repeatedly.
We may label these elements as $(e_1,e_1',e_2,e_2',\ldots,e_T,e_T')$, for some $T>0$.
Note that $T=4n^3+O(n^2)$ by construction.\COMMENT{Each $A_{i,j}$ corresponds to exactly $n+2(j-1)$ steps of the process, while each $A_{i,-j}$ corresponds to exactly $n+2j-1$.
For fixed $i$, each one of these, except the last one, is extended by one step when concatenating them.
This is repeated for each of the possible values of $i$, and these are concatenated by using $3$ more steps.
Therefore, the total number of steps is
\[n\left(2n-2+\sum_{j=0}^{2n-2}(n+j)\right)+3(n-1)=n\left(2n-2+(2n-1)n+(2n-1)(n-1)\right)+3(n-1)=4n^3+O(n^2).\]}
Note, moreover, that by construction we are guaranteed that $|\{e_t,e_t'\}\cap A_{i,j}|\in\{0,2\}$ for all $i$ and $j$.
Additionally, any two consecutive triples in $A$ share exactly two vertices, thus, it is easy to check that any three consecutive triples span five vertices.

Let $e_0'\coloneqq t_1m_0b_1$.
For each $t\in[T-1]$, let $H_t$ denote the copy of $K_5^{(3)}$ with vertex set $e_{t-1}'\cup e_t\cup e_t'$, and let $H_t'$ denote the copy of $K_5^{(3)}$ with vertex set $e_{t-1}\cup e_{t-1}'\cup e_t$ (if $t>1$).
For each $t\in[T]$, let $G_t$ be the hypergraph with edge-set $E(G_{t-1})\cup\{e_t,e_t'\}$.
We will show that these graphs indeed coincide with those obtained by the $K_5^{(3)}$-bootstrap percolation process with initial infection $G_0$.

\begin{claim}\label{claim:2edges}
Let\/ $H$ be a copy of\/ $K_5^{(3)}$ on the vertex set of\/ $G_0$.
Assume that, for some\/ $t\in[0,T-1]$, we have that\/ $H\nsubseteq G_t$ but\/ $H$ is\/ $2$-completable in\/ $G_t$.
Then, the following hold.
\begin{itemize}
    \item If\/ $H$ is\/ $1$-completable in\/ $G_t$, suppose that adding\/ $e$ to\/ $G_t$ completes\/ $H$.
    Then,\/ $t\geq 1$,\/ $e=e_{t+1}$ and\/ $H=H_{t+1}'$.
    \item If\/ $H$ is not\/ $1$-completable in\/ $G_t$, suppose that adding\/ $e$ and\/ $e'$ to\/ $G_t$ completes\/ $H$.
    Then,\/ $\{e,e'\}=\{e_{t+1},e_{t+1}'\}$ and\/ $H=H_{t+1}$.
\end{itemize}
\end{claim}

\begin{claimproof}
Consider any copy $H$ of $K_5^{(3)}$ on $V(G_0)$.
If $H$ contains two vertices of the form $t_i$ and $t_{i'}$ with $1\leq i< i'\leq n$, since $G_T$ does not contain any edge with two `top' vertices (see \ref{initial1}--\ref{initial11} as well as \eqref{eq:2perc1}--\eqref{eq:2perc4}), $H$ must be missing at least three edges in $G_T$.\COMMENT{One for each other vertex in $H$.} 
As $G_0\subseteq G_1\subseteq\ldots\subseteq G_T$, it follows that $H$ is not $2$-completable in $G_t$ for any $t\in[0,T-1]$.
The same argument holds if $H$ contains two vertices of the form $b_j$ and $b_{j'}$.
Hence, we may assume that $H$ contains at most one vertex $t_i$ and one vertex $b_j$.
Similarly, if $H$ contains two vertices $m_\ell$ and $m_{\ell'}$ with $|\ell-\ell'|\geq3$, then $G_T$ does not contain any edge containing both $m_\ell$ and $m_{\ell'}$, so $H$ is not $2$-completable in $G_T$.
Therefore, $H$ contains at most three vertices of the form $m_\ell$, and their indices must be within distance two of each other.

Assume first that $H$ contains some vertex of the form $d_{i,c}$ with $i\in[n-1]$ and $c\in[3]$.
All triples in $G_T$ containing one such vertex are of the form $d_{i,r}d_{i,p}d_{i,q}$ with $-2\leq r<p<q\leq6$ and $q-r\leq4$ (see \ref{initial11} and \eqref{eq:2perc4}).
It follows easily that, if $V(H)$ does not consist of five consecutive vertices $d_{i,p},d_{i,p+1},\ldots,d_{i,p+4}$ with $-2\leq p\leq2$, then $H$ cannot be $2$-completable in $G_T$.\COMMENT{Say $H$ contains the vertex $d_{i,c}$ and some other vertex $v$ which does not satisfy what is said above. Then, $G_0$ does not contain any edge containing both $d_{i,c}$ and $v$, so at least three edges of $H$ are missing.}
Moreover, if we assume $V(H)=\{d_{i,p+h}:0\leq h\leq4\}$ for some $i\in[n-1]$ and $p\in\{-1,1\}$, then we also know that the triples $d_{i,p}d_{i,p+1}d_{i,p+4}$, $d_{i,p}d_{i,p+2}d_{i,p+4}$ and $d_{i,p}d_{i,p+3}d_{i,p+4}$ do not appear in $G_T$, so again $H$ cannot be $2$-completable.
So we must have $V(H)=\{d_{i,p+h}:0\leq h\leq4\}$ with $p\in\{-2,0,2\}$.
But then, by \ref{initial11} and \eqref{eq:2perc4}, the only three edges missing from $H$ in $G_0$ are $e_{N-1}'$, $e_N$ and $e_N'$, for some $N\in[T]$.
Let $t\in[0,T-1]$ be such that $H$ is $2$-completable in $G_t$ but $E(H)\not \subseteq E(G_t)$.
It is easy to see that we must have $t=N-1$; furthermore, $H=H_{t+1}$, $H$ is not $1$-completable in $G_t$, and $E(G_{t+1})\setminus E(H)=\{e_N,e_N'\}$, as desired.

Assume next that $H$ does not contain any vertex of the form $d_{i,c}$ with $i\in[n-1]$ and $c\in[3]$, so it must contain one top vertex $t_i$, one bottom vertex $b_j$, and three consecutive middle vertices $m_\ell$, $m_{\ell+1}$ and $m_{\ell+2}$.
Assume $j>0$ (the other case can be argued analogously).\COMMENT{In fact, the other case is simpler, as in this case we have ``boundary issues''.}
If $\ell\geq2(n+j)-1$, then $G_T$ is missing the edges $t_im_{\ell+2}b_j$, $b_jm_{\ell+1}m_{\ell+2}$ and $b_jm_{\ell}m_{\ell+2}$ (see \ref{initial4}, \ref{initial7}, \ref{initial9}, \eqref{eq:2perc1} and \eqref{eq:2perc3}), so $H$ cannot be $2$-completable at any stage of the process.
Similarly, if $\ell<-2(j-1)$, then $G_T$ is missing the edges $t_im_\ell b_j$, $b_jm_{\ell}m_{\ell+1}$ and $b_jm_{\ell}m_{\ell+2}$ (see \ref{initial4}, \ref{initial7}, \ref{initial10}, \eqref{eq:2perc1} and \eqref{eq:2perc3}), hence $H$ is not $2$-completable.
Thus, we must have $-2(j-1)\leq\ell\leq2(n+j-1)$.
However, for the case when $j=n$ and $\ell\in\{4n-3,4n-2\}$, it follows from \ref{initial2}, \ref{initial3}, \ref{initial6} and \eqref{eq:2perc3} that $G_T$ is missing the triples $m_{\ell}m_{\ell+1}m_{\ell+2}$, $t_im_{\ell}m_{\ell+2}$ and $t_im_{\ell+1}m_{\ell+2}$, hence $H$ cannot be $2$-completable.
So we must have $-2(j-1)\leq\ell\leq2(n+j-1)$ when $j\in[n-1]$ and $-2(j-1)\leq\ell\leq2(n+j-2)$ when $j=n$.
Let $t\in[0,T-1]$ be such that $H$ is $2$-completable in $G_t$ but $E(H_t)\not \subseteq E(G_t)$.
We now split the analysis into further cases.

Assume first that $j<n$ and $\ell=2(n+j-1)$.
It follows from \ref{initial2}--\ref{initial10} that the only triples of $H$ missing from $G_0$ are $t_im_{\ell}b_j$, $t_im_\ell m_{\ell+1}$ and $t_im_{\ell+1}m_{\ell+2}$.
These three are added throughout the sequence of edges defined above, as follows from \eqref{eq:2perc1} and \eqref{eq:2perc3}, as $e_{N-1}'$, $e_N$ and $e_N'$, respectively, for some $N\in[T]$.
Then, in order for $H$ to be $2$-completable in $G_t$, we must have $e_{N-1}'\in E(G_t)$, and hence $t=N$, $H$=$H_{t+1}$, $H$ is not $1$-completable, and $E(G_{t+1})\setminus E(H)=\{e_N,e_N'\}$, as desired.

Consider next the case that $j<n$ and $\ell=2(n+j-1)-1$.
The triples of $H$ missing from $G_0$ are $t_im_{\ell}b_j$, $t_im_{\ell+1}b_j$ and $t_im_{\ell+1}m_{\ell+2}$ (see \ref{initial2}--\ref{initial10}), which are $e_{N-1}$, $e_{N-1}'$ and $e_N$, respectively, for some $N\in[T]$ (see \eqref{eq:2perc1} and \eqref{eq:2perc3}).
Thus, in order for $H$ to be $2$-completable in $G_t$, this graph must contain at least one of the missing triples; however, since $e_{N-1}$ and $e_{N-1}'$ are added simultaneously in the sequence of graphs, we must have $e_{N-1},e_{N-1}'\in E(G_t)$, and so $H$ is $1$-completable.
Then, the only edge that can complete $H$ is $e=e_N$, and so it follows that $N=t+1$ and $H=H_{t+1}'$.

Assume now that $\ell=-2(j-1)$.
Here we have two further subcases.
If $j=1$, then the only triples of $H$ missing in $G_0$ are precisely $e_1$ and $e_1'$ (see \ref{initial1}--\ref{initial10} as well as \eqref{eq:2perc1}).
Therefore, we have $\{e,e'\}=\{e_1,e_1'\}$ and $H=H_{1}$.
So suppose that $j\geq2$.
Then, the triples of $H$ missing in $G_0$ are $t_im_\ell m_{\ell+1}$, $t_im_{\ell+1}m_{\ell+2}$, $t_im_{\ell+1}b_j$ and $t_im_{\ell+2}b_j$ (see \ref{initial2}--\ref{initial10}).
But then it follows from \eqref{eq:2perc1} and \eqref{eq:2perc3} that these triples take the form $e_{N-1}$, $e_{N-1}'$, $e_N$, $e_N'$, for some $N\in[T]$.
In order for $H$ to be $2$-completable in $G_t$, we must have $e_{N-1},e_{N-1}'\in E(G_t)$.
Then, it follows that $t=N-1$, $H=H_{t+1}$, $H$ is not $1$-completable, and  $E(G_{t+1})\setminus E(H)=\{e_N,e_N'\}$.

Suppose now that $\ell=-2(j-1)+1$.
Again, we must consider two subcases.
If $j=1$, the edges of $H$ missing in $G_0$ are $t_im_{\ell}b_j$, $t_im_{\ell+1}b_j$ and $t_im_{\ell+2}b_j$, which correspond to $e_1$, $e_1'$ and $e_2$ (see \ref{initial2}--\ref{initial10} as well as \eqref{eq:2perc1}).
Thus, in order for $H$ to be $2$-completable in $G_t$ we must have $e_1,e_1'\in E(G_t)$, so $t=1$ and $H$ is $1$-completable in $G_1$.
It then follows that $e=e_2$, and $H=H_2'$.
So suppose $j\geq2$.
Then, the triples of $H$ missing in $G_0$ are $t_im_\ell m_{\ell+1}$, $t_im_{\ell}b_j$, $t_im_{\ell+1}b_j$ and $t_im_{\ell+2}b_j$ (see \ref{initial2}--\ref{initial10}).
By \eqref{eq:2perc1} and \eqref{eq:2perc3}, these triples take the form $e_{N-2}'$, $e_{N-1}$, $e_{N-1}'$, $e_N$, for some $N\in[2,T]$.
In order for $H$ to be $2$-completable in $G_t$, at least two of these edges must be added.
But $e_{N-1}$ and $e_{N-1}'$ are added simultaneously, so we conclude that $e_{N-2}',e_{N-1},e_{N-1}'\in E(G_t)$ and $H$ is $1$-completable in $G_t$.
Therefore, $E(G_{t+1})\setminus E(H)=\{e_N\}$, $N=t+1$ and $H=H_{t+1}'$.

Suppose finally that $-2(j-2)\leq\ell\leq2(n+j-2)$.
Then, the only edges of $H$ missing in $G_0$ are $t_im_{\ell}b_j$, $t_im_{\ell+1}b_j$ and $t_im_{\ell+2}b_j$ (see \ref{initial2}--\ref{initial10}).
If $\ell$ is even, it follows from \eqref{eq:2perc1} that these edges take the form $e_{N-1}'$, $e_N$ and $e_N'$, respectively, for some $N\in[T]$; on the contrary, if $\ell$ is odd, then they take the form $e_{N-1}$, $e_{N-1}'$ and $e_N$.
In the former case, in order for $H$ to be $2$-completable in $G_t$, we must have $e_{N-1}'\in E(G_t)$, and it follows that $E(G_{t+1})\setminus E(H)=\{e_N,e_N'\}$, $N=t+1$ and $H=H_{t+1}$.
In the latter case, we must have $e_{N-1},e_{N-1}'\in E(G_t)$, so $H$ is $1$-completable, and it follows that $E(G_{t+1})\setminus E(H)=\{e_N\}$, $N=t+1$ and $H=H_{t+1}'$.
\end{claimproof}

By applying \cref{claim:2edges} iteratively, we conclude that the $(K_5^{(3)},2)$-bootstrap percolation process with initial infection $G_0$ indeed generates the sequence of graphs $G_0,G_1,\ldots,G_T,\ldots$, so its running time is at least $T=4n^3+O(n^2)$.
By taking into account the number of vertices of the graphs we are considering, we conclude that $M_{(K_{5}^{(3)},2)}(n)\geq4n^3/12^3+O(n^2)$.
\end{proof}

\section{Concluding remarks}

Graph and hypergraph bootstrap percolation have seen a lot of research in recent years, with many intriguing questions remaining open and many possible avenues for further research.
We have focused particularly on understanding the maximum running time of these processes.
Our first main result, \cref{theorem_bootstrapgeneral}, building on the previous work of \citet{NR22}, has allowed us to conclude that the maximum running time of $K_{k}^{(r)}$-bootstrap percolation is of order $\Theta(n^r)$ for any $k>r\geq3$.
A first very natural problem is to determine the leading constant in this asymptotic behaviour.

\begin{problem}
For each\/ $k>r\geq2$, determine the limit of\/ $M_k^r(n)/n^r$.
\end{problem}

In particular, all results in this hypergraph context have relied on the trivial upper bound that $M_H^r(n)\leq\inbinom{n}{r}$; obtaining better upper bounds should be the first step towards this problem.
We also note that the lower bound arising from our construction (see \cref{rem:bound1}) is not tight.

Another very natural direction is to study the asymptotic growth of $M_H(n)$ when $H$ is an $r$-uniform hypergraph which is not complete.
We have made the first progress in this direction by addressing two particular cases, see Theorems~\ref{thmp:1}' and~\ref{thmp:2}'.
A more general study of this problem for different instances of $H$ is crucial towards a unified understanding of hypergraph bootstrap percolation.

More generally, the notion of more `powerful' infections that we proposed when considering $(H,m)$-bootstrap percolation leads to many new open problems.
Here we have only addressed two particular instances to showcase that this notion leads to interesting results.
In the case of $(K_5^{(3)},2)$-bootstrap percolation, \cref{thm:cubic} shows that the maximum running time is cubic, that is, as large as it could possibly be (up to constant factors).
In the case of $(K_4^{(3)},2)$-bootstrap percolation, however, the maximum running time is only linear, and in \cref{prop:linear} we have determined the exact value of this maximum running time for all values of $n$.
Remarkably, this is the only nontrivial exact result in the area other than those for $K_3$- and $K_4$-bootstrap percolation~\cite{BPRS17}.
It would certainly be desirable to understand the behaviour of the maximum running time of $(H,m)$-bootstrap percolation more generally.
To begin, we propose the following problem.

\begin{problem}
Given\/ $k>r\geq2$ and\/ $m\in[\inbinom{k}{r}]$, determine the asymptotic behaviour of the maximum running time of the\/ $(K_k^{(r)},m)$-bootstrap percolation process.
\end{problem}

It would also be interesting to consider this more general $(H,m)$-bootstrap percolation process in other contexts where graph bootstrap percolation has been studied.
In particular, one may consider the extremal problem, i.e., what is the minimum number of edges an initial $r$-uniform infection $G_0$ on $n$ vertices can have if we know the $(H,m)$-percolation process $G_0,\ldots,G_T$ satisfies $G_T=K_n^{(r)}$?

\bibliographystyle{mystyle} 
\bibliography{bootstrap}


\end{document}